\documentclass[reqno]{amsart}
\pdfoutput=1

\usepackage{latexsym}   
\usepackage{amssymb,accents}    
\usepackage[usenames,dvipsnames]{xcolor} 
\usepackage[colorlinks, 
            linkcolor=MidnightBlue, 
            citecolor=MidnightBlue, 
            urlcolor =BlueViolet,
           ]{hyperref}
\usepackage{booktabs}
\usepackage[all]{xy}
\usepackage{colonequals}
\usepackage{enumitem}

\newtheorem{theorem}{Theorem}[section]
\newtheorem{lemma}[theorem]{Lemma}
\newtheorem{proposition}[theorem]{Proposition}
\newtheorem{corollary}[theorem]{Corollary}

\theoremstyle{definition}
\newtheorem{algorithm}[theorem]{Algorithm}
\newtheorem{definition}[theorem]{Definition}

\theoremstyle{remark}
\newtheorem{remark}[theorem]{Remark}

\newlength{\algindent}
\newcounter{stepcount}

\newenvironment{alglist}
{\quad\begin{list}{\arabic{stepcount}.}%
{\leftmargin=\algindent\labelwidth=\algindent\itemsep=\smallskipamount\usecounter{stepcount}}}
{\end{list}}

\newcommand{\algin}{\item[\emph{Input}:]}
\newcommand{\algout}{\item[\emph{Output}:]}

\newenvironment{algsublist}
{\begin{enumerate}[label=({\alph*}), ref=({\alph*})]}
{\end{enumerate}}

\newenvironment{algsubsublist}
{\begin{enumerate}[label=({\roman*}), ref=({\roman*})]}
{\end{enumerate}}

\DeclareMathOperator{\Aut}{Aut}
\DeclareMathOperator{\Cross}{Cross}
\DeclareMathOperator{\Div}{div}
\DeclareMathOperator{\Gal}{Gal}
\DeclareMathOperator{\Norm}{Norm}
\DeclareMathOperator{\PGL}{PGL}
\DeclareMathOperator{\Proj}{Proj}
\DeclareMathOperator{\Supp}{Supp}
\newcommand{\FF}{{\mathbf F}}
\newcommand{\PP}{{\mathbf P}}

\newcommand{\CF}{{\mathcal F}}
\newcommand{\CL}{{\mathcal L}}
\newcommand{\CP}{{\mathcal P}}

\newcommand{\kbar}{\bar{k}}
\newcommand{\Fbar}{\overline{\FF}}

\newcommand{\Otilde}{\lowerwidetilde{O}}

\newcommand{\twobytwo}[4]{
\bigl[\begin{smallmatrix}#1&#2\\#3&#4\end{smallmatrix}\bigr]}

\newcommand\proj[2]{[#1\,{:}\,#2]}

\newcommand{\kstar}{k^\times}
\newcommand{\Rhom}{R_\textup{hom}}

\newcommand{\zz}{\phantom{.00}}             

\makeatletter
\newcommand{\raisemath}[1]{\mathpalette{\raisem@th{#1}}}
\newcommand{\raisem@th}[3]{\raisebox{#1}{$#2#3$}}
\makeatother

\newcommand{\Pover}[1]{\PP^1\phantom{#1}\llap{$\raisemath{-0.4ex}{\scriptstyle/#1}$}}

\newcommand{\mybar}[1]{
  \mathchoice
  {#1\llap{$\overline{\phantom{\displaystyle\rm#1}}$}}
  {#1\llap{$\overline{\phantom{\textstyle\rm#1}}$}}
  {#1\llap{$\overline{\phantom{\scriptstyle\rm#1}}$}}
  {#1\llap{$\overline{\phantom{\scriptscriptstyle\rm#1}}$}}
}  
\renewcommand{\bar}{\mybar}

\DeclareMathSymbol{\widetildesym}{\mathord}{largesymbols}{"65}
\newcommand\lowerwidetildesym{%
  \text{\smash{\raisebox{-1.4ex}{%
    $\widetildesym$}}}}
        \newcommand\lowerwidetilde[1]{%
          \mathchoice
            {\accentset{\displaystyle\lowerwidetildesym}{#1}}
            {\accentset{\textstyle\lowerwidetildesym}{#1}}
            {\accentset{\scriptstyle\lowerwidetildesym}{#1}}
            {\accentset{\scriptscriptstyle\lowerwidetildesym}{#1}}
}

\begin{document}

\title[Enumerating places in quasilinear time]
      {Enumerating places of \texorpdfstring{$\mathbf P^1$}{P1} up to \\automorphisms of \texorpdfstring{$\mathbf P^1$}{P1} in quasilinear time}
\author{Everett W. Howe}
\address{Independent mathematician, 
         San Diego, CA 92104, USA}
\email{\href{mailto:however@alumni.caltech.edu}{however@alumni.caltech.edu}}
\urladdr{\href{http://ewhowe.com}{http://ewhowe.com}}

\date{21 October 2025}



\begin{abstract}
We present an algorithm that, for every fixed degree $n\ge 3$, will enumerate
all degree-$n$ places of the projective line over a finite field $k$ up to the 
natural action of $\operatorname{PGL}_2(k)$ using $O(\log q)$ space and  
$\widetilde{O}(q^{n-3})$ time, where ${q=\#k}$. Since there are $\Theta(q^{n-3})$
orbits of $\operatorname{PGL}_2(k)$ acting on the set of degree\nobreakdash-$n$
places, the algorithm is quasilinear in the size of its output. The algorithm is
probabilistic unless we assume the extended Riemann hypothesis.

We also present an algorithm for enumerating orbit representatives for the 
action of $\operatorname{PGL}_2(k)$ on the degree-$n$ effective divisors of
$\mathbf{P}^1$ over finite fields~$k$. The two algorithms depend on one another;
our method of enumerating orbits of places of odd degree~$n$ depends on
enumerating orbits of effective divisors of degree~$(n+1)/2$.

As an application of the second algorithm, for $g=2$, $3$, and $4$ we implement
an algorithm in Magma that computes all hyperelliptic curves of genus $g$ over
finite fields $k$ using $O(q^{g-1})$ space and $\widetilde{O}(q^{2g-1})$ time,
where $q=\#k$. Our implementation runs $60$--$80$ times faster than existing
algorithms for computing genus-$2$ hyperelliptic curves, and about $280$ times
faster than existing algorithms for computing genus-$3$ hyperelliptic curves. We
know of no other implementations of algorithms to compute genus-$4$
hyperelliptic curves.
\end{abstract}

\maketitle

\section{Introduction}
\label{S:intro}

The effective divisors of $\PP^1$ are a fundamental object in geometry, and 
there are situations in which one would like to enumerate all such divisors of a
given degree when the base field $k$ is finite, up to the action of 
$\Aut(\Pover{k})\cong\PGL_2(k)$. For instance, if one is interested in 
enumerating hyperelliptic curves of genus $g>1$ over~$\FF_q$, where $q$ is an 
odd prime power, then it is natural to look at effective divisors of degree
$2g+2$ in which no point has multiplicity larger than~$1$, up to the action 
of~$\PGL_2(\FF_q)$ (see~\cite[\S3.1]{MaisnerNart2002}, \cite[\S0]{Nart2009}).
More generally, if one is interested in covers of $\Pover{k}$ of a given degree
and genus up to isomorphism, a natural place to start is to group them by 
ramification divisor, up to the action of~$\PGL_2(k)$.

An argument that we give in Section~\ref{S:divisors} shows that to develop an 
algorithm to enumerate all degree-$n$ effective divisors of $\Pover{\FF_q}$ up 
to $\PGL_2(\FF_q)$ in quasilinear time, the central case to consider is that of
effective divisors consisting of a single place of degree~$n$. When $n\le 3$
this is trivial, and Theorems~1.2 and~3.8 of~\cite{Howe2024} tell us how to do 
this when $n=4$. Our main result, therefore, is the following:

\begin{theorem}
\label{T:places}
Fix an integer $n\ge 5.$ The algorithms we present in 
Sections~\textup{\ref{S:simpleodd}} and~\textup{\ref{S:simpleeven}} take as
input a prime power $q$ and output a complete set of unique representatives for
the orbits of $\PGL_2(\FF_q)$ acting on the degree-$n$ places of 
$\Pover{\FF_q}$. The algorithms take time $\Otilde(q^{n-3})$ and space
$O(\log q)$, and are probabilistic unless we assume the extended Riemann 
hypothesis.
\end{theorem}

\begin{corollary}
\label{C:divisors}
Fix an integer $n\ge 3.$ The algorithm we present in 
Section~\textup{\ref{S:divisors}} takes as input a prime power $q$ and outputs a
complete set of unique representatives for the orbits of $\PGL_2(\FF_q)$ acting 
on the effective divisors of $\Pover{\FF_q}$ of degree~$n$. The algorithm takes
time $\Otilde(q^{n-3})$ and space $O(\log q)$, and is probabilistic unless we
assume the extended Riemann hypothesis.
\end{corollary}

(Here and throughout the paper, by a ``complete set of unique representatives''
for a group acting on a set we mean a set of orbit representatives that contains 
exactly one element from each orbit.)

Since up to $\PGL_2(\FF_q)$ there are roughly $q^{n-3}/n$ degree-$n$ places of
$\PP^1$ and roughly $q^{n-3}$ effective divisors of degree~$n$, we see that both
algorithms take time quasilinear in the size of their output.

Almost all of the ingredients of our algorithms already appear
in~\cite{Howe2024}, which focuses on enumerating a complete set of unique 
representatives for the action of $\PGL_2$ on the effective divisors of $\PP^1$
of even degree $n$ that do not include any places with multiplicity greater
than~$1$. However, the restriction to even $n$ (or, more generally, 
composite~$n$) was critical to the approach in~\cite{Howe2024}. Furthermore, the
algorithms in that paper do not attempt to minimize their memory usage.

The main new idea in this paper is a method for enumerating $\PGL_2$ orbits of
places of odd degree. Given an odd $n\ge 3$, we show how to assign to every
degree-$n$ place $P$ of $\Pover{\FF_q}$ a rational function on $\Pover{\FF_q}$
of degree at most $(n-1)/2$ that reflects the action of Frobenius on the
geometric points of $P$; we call this the \emph{Frobenius function} for~$P$. If 
$F$ is the Frobenius function for~$P$, then its divisor $D$ of fixed points is 
effective and has degree at most $(n+1)/2$; we call $D$ the 
\emph{Frobenius divisor} of~$P$. The map that sends a degree-$n$ place $P$ to 
its Frobenius divisor is equivariant with respect to the action of 
$\PGL_2(\FF_q)$ on places and divisors. Also, there is a straightforward way to
enumerate all degree-$n$ places with a given Frobenius divisor. Thus, to 
enumerate all degree-$n$ places up to the action of $\PGL_2(\FF_q)$, we simply
enumerate orbit representatives for the action of $\PGL_2(\FF_q)$ on effective
divisors of degree at most $(n+1)/2$, and then find all degree-$n$ places whose
Frobenius divisors are among these orbit representatives. We can compute the
orbit representatives of effective divisors of degree at most $(n+1)/2$ by 
recursively applying our algorithm --- or by other means, since any algorithm
that takes time at most $\Otilde(q^{n-3})$ and space at most $O(\log q)$ will
suffice.

The critical step in computing the places with a given Frobenius divisor 
involves finding the degree-$n$ factors of a polynomial of degree bounded by
$1 + ((n-1)/2)^n$, so the complexity of our algorithm is essentially tied to 
that of the polynomial factorization algorithm we use. Under the extended 
Riemann hypothesis, we can deterministically factor polynomials in $\FF_q[x]$ of
bounded degree in time polynomial in $\log q$ 
(see~\cite{Ronyai1992}, \cite{Evdokimov1994}); unconditionally, there are 
probabilistic algorithms that factor in polynomial time
(see~\cite{Berlekamp1970}, \cite{CantorZassenhaus1981}, 
\cite{vonzurGathenShoup1992}, \cite{KaltofenShoup1998}, 
\cite{KedlayaUmans2011}). Thus, a probabilistic version of our algorithm takes
time $\Otilde(q^{n-3})$, and a deterministic version has the same complexity if
we assume the extended Riemann hypothesis.

In Section~\ref{S:ingredients} we present the main ingredients of our argument,
most notably the idea of the Frobenius function of a place of odd degree, as
well as how such rational functions are related to their divisors of fixed 
points. In Section~\ref{S:algorithm} we observe that a low-memory algorithm to 
output a finite list can be converted into a low-memory algorithm that takes one
element of the list (plus some extra data) as input and outputs the next element
(plus the necessary extra data). This can be done in such a way that the time to
create the whole list by repeatedly running the second algorithm is essentially
the same as that of running the first algorithm once. This observation
simplifies our exposition of our main algorithms.

In Section~\ref{S:structure} we outline the recursive nature of our algorithms
and explain how the arguments in the following sections piece together to
create a proof of Theorem~\ref{T:places}. In Section~\ref{S:divisors} we show 
how to enumerate $\PGL_2(\FF_q)$ orbits of effective divisors of degree~$n$ if
we can enumerate the $\PGL_2(\FF_q)$ orbits of places of degree at most~$n$, 
which shows that Corollary~\ref{C:divisors} follows from Theorem~\ref{T:places}.
In Section~\ref{S:simpleodd} we present an algorithm to enumerate orbit
representatives for $\PGL_2(\FF_q)$ acting on places of odd degree by using the
connection between a place and its Frobenius divisor. In 
Section~\ref{S:simpleeven} we show how to handle places of even degree $n$ over 
$\FF_q$ by combining our algorithm for places of degree $n/2$ over $\FF_{q^2}$
with an explicit list of coset representatives for $\PGL_2(\FF_q)$ in
$\PGL_2(\FF_{q^2})$. Finally, in Section~\ref{S:implementation} we give timings 
for some Magma programs for enumerating hyperelliptic curves of genus at most
$4$ based on the ideas in this paper.

\section{Mise en place}
\label{S:ingredients}

In this section we lay out the ingredients necessary to prove our main results. 
We start by defining the basic objects of study, and then in separate 
subsections we present various results needed for our later arguments.

Let $k$ be a finite field and let $R\colonequals k[x,y]$. We view $R$ as a
graded ring, graded by degree, and we view $\Pover{k}$ as $\Proj R$. For each
$n$ let $R_{n}$ be the set of homogeneous polynomials in $R$ of degree $n$, and
let $\Rhom$ be the union of the $R_{n}$. We say that $f(x,y)\in \Rhom$ is
\emph{monic} if $f(x,1)$ is monic as a univariate polynomial. We say that
$\alpha\in\kbar$ is a \emph{root} of $f$ if $f(\alpha,1) = 0$.

There is a natural action of the group $\PGL_2(k)$ on $\PP^1(\kbar)$: If
$\Gamma\in\PGL_2(k)$ is represented by a matrix $\twobytwo{a}{b}{c}{d}$ and if
$P\colonequals\proj{\alpha}{\beta}\in\PP^1(\kbar)$, then we define
\[
\Gamma(P)\colonequals\proj{a\alpha+b\beta}{c\alpha+d\beta}.
\]
This action extends to an action on divisors on $\Pover{k}$.

There is a related action of $\PGL_2(k)$ on the set $\Rhom/\kstar$, such that 
$\Gamma$ sends the class $[f]$ of a homogeneous polynomial $f(x,y)$ to the class
of ${f(dx-by,-cx+ay)}$. Since every element of $\Rhom/\kstar$ has a unique monic
representative, we can also view $\PGL_2(k)$ as acting on the set of homogeneous
polynomials; given a homogenous~$f$, we let $\Gamma(f)$ be the unique monic $f'$
such that $[f']=\Gamma([f])$. Let $\Div$ be the function that sends a homogenous
polynomial to its divisor. Then we have 
\begin{equation}
\label{EQ:divgamma=gammadiv}
\Div\Gamma(f) =\Gamma(\Div f).
\end{equation}

Our goal is to produce, for each $n>3$, an algorithm that will create a complete
set of unique representatives for the action of $\PGL_2(k)$ on the degree-$n$ 
effective divisors of $\Pover{k}$, in time $\Otilde(q^{n-3})$, where $q = \#k$.
Since effective divisors on $\Pover{k}$ correspond to monic elements of~$\Rhom$,
our goal can also be viewed as finding orbit representatives for $\PGL_2(k)$
acting on the monic elements of $R_{n}$.

\subsection{Frobenius functions and Frobenius divisors}
\label{SS:Frobenius}
In this subsection we introduce the concepts of Frobenius functions and
Frobenius divisors and prove some basic results about them.

\begin{theorem}
\label{T:Frobeniusfunction}
Let $f\in k[x,y]$ be an irreducible homogeneous polynomial of degree~$n$, where
$n\ge 3$ is odd. Among the rational functions $\PP^1\to\PP^1$ of degree at most
$(n-1)/2$, there is a unique function $F$ such that 
$\proj{\alpha^{\#k}}{1} = F(\proj{\alpha}{1})$ for every root $\alpha$ of $f$ 
in~$\kbar$. 
\end{theorem}

\begin{definition}
Given $f$ and $F$ as in the theorem, we say that $F$ is the 
\emph{Frobenius function} for $f$. If $P$ is the degree-$n$ place corresponding
to $f$, we also say that $F$ is the Frobenius function for~$P$.
\end{definition}

\begin{proof}[Proof of Theorem~\textup{\ref{T:Frobeniusfunction}}]
Let $q = \# k$, let $r = (n-1)/2$, and let $\alpha\in\FF_{q^n}$ be a root 
of~$f$. Then 
\[
       1,       \alpha, \alpha^2, \ldots, \alpha^r, 
\alpha^q, \alpha^{q+1},           \ldots, \alpha^{q+r}
\]
is a list of $2r+2 = n+1$ elements of $\FF_{q^n}$, so there is an $\FF_q$-linear
relation among these elements, say
\[
a_0 + a_1\alpha + \cdots + a_r\alpha^r 
    = b_0\alpha^q + b_1\alpha^{q+1} + \cdots + b_r\alpha^{q+r}\,.
\]
We know that the two sides of this equality are nonzero, because otherwise
$\alpha$ would be a root of a polynomial of degree smaller than~$n$. Therefore 
we can write
\[
\alpha^q = \frac{a_r\alpha^r + \cdots + a_1\alpha + a_0}
                {b_r\alpha^r + \cdots + b_1\alpha + b_0}\,.
\]
Let $F$ be the rational function
\[
F \colonequals \frac{a_r x^r + \cdots + a_1 x y^{r-1} + a_0 y^r}
                    {b_r x^r + \cdots + b_1 x y^{r-1} + b_0 y^r}\,.
\]
Then $\proj{\alpha^q}{1} = F(\proj{\alpha}{1})$, and the same holds if we
replace $\alpha$ by any of its conjugates. Also, $F$ is a rational function of
degree at most~$r$ (``at most'' because there may be common factors in the 
numerator and denominator that cancel one another). This proves the existence
claim of the theorem.

Suppose $F_1$ and $F_2$ are two rational functions of degree at most~$r$ such 
that $\proj{\alpha^{q}}{1} = F_i(\proj{\alpha}{1})$ for $i=1$ and $i=2$. Write
$F_i = g_i/h_i$ for homogeneous polynomials $g_i$ and $h_i$ of degree at
most~$r$. Then since $F_1(\proj{\alpha}{1}) = F_2(\proj{\alpha}{1})$, we see
that $\proj{\alpha}{1}$ is zero of the homogeneous polynomial $g_1h_2 - g_2h_1$,
whose degree is at most $2r < n$. Since $\alpha$ generates a degree-$n$ 
extension of $\FF_q$, this is impossible unless $g_1h_2 = g_2h_1$; that is,
unless $F_1 = F_2$. This proves the uniqueness claim of the theorem.
\end{proof}

\begin{definition}
Let $F$ be a rational function of degree $r$ on $\Pover{k}$ and write $F = g/h$ 
for coprime homogenous polynomials $g$ and $h$ in $R_{r}$. If $F\ne x/y$, we 
define the \emph{divisor of fixed points} for $F$ to be the divisor
$\Phi(F)\colonequals \Div (xh-yg)$ of degree~$r+1$. If $f$ is an irreducible
homogeneous polynomial whose Frobenius function is~$F$, then we call $\Phi(F)$
the \emph{Frobenius divisor} for~$f$.
\end{definition}

\begin{lemma}
\label{L:PGLequivariance}
Let $f$ be a monic irreducible homogeneous polynomial in $R$ of odd degree 
$n\ge 3$ and let $F$ be the Frobenius function for $f$. For every
$\Gamma\in\PGL_2(k)$, the rational function $\Gamma\circ F\circ \Gamma^{-1}$ is
the Frobenius function for $\Gamma(f)$, and
\begin{equation}
\label{EQ:gammaphi}
\Phi(\Gamma\circ F\circ \Gamma^{-1}) = \Gamma(\Phi(F))\,.
\end{equation}
\end{lemma}

\begin{proof}
If $\beta$ is a root of $\Gamma(f)$ then we have 
$\proj{\beta}{1} = \Gamma(\proj{\alpha}{1})$ for a root $\alpha$ of~$f$. We 
check that
\begin{align*}
\proj{\beta^{\#k}}{1} 
&= \Gamma(\proj{\alpha^{\#k}}{1}) \\
&= \Gamma(F(\proj{\alpha}{1})) \\
&= (\Gamma\circ F)(\proj{\alpha}{1}) \\
&= (\Gamma\circ F\circ\Gamma^{-1})(\Gamma(\proj{\alpha}{1}))\\
&= (\Gamma\circ F\circ\Gamma^{-1})(\proj{\beta}{1})\,,
\end{align*}
so $\Gamma\circ F\circ \Gamma^{-1}$ is the Frobenius function for $\Gamma(f)$.
Suppose $\twobytwo{a}{b}{c}{d}$ is a matrix that represents $\Gamma$, and
suppose $g$ and $h$ are coprime homogenous polynomials such that $F = g/h$.
Using the definition of $\Phi$, we find that 
$\Phi(\Gamma\circ F\circ \Gamma^{-1})$ is the divisor of the polynomial 
\[
x\cdot [cg(x',y') + dh(x',y')] - y \cdot [ag(x',y') + bh(x',y')]\,,
\]
where $x' = dx-by$ and $y'=-cx+ay$. But this simplifies to
\[
x'\cdot h(x',y') - y'\cdot g(x',y')\,,
\]
and the class of this polynomial in $\Rhom/\kstar$ is nothing other than 
$\Gamma([xh-yg])$. By~\eqref{EQ:divgamma=gammadiv}, we see 
that~\eqref{EQ:gammaphi} holds.
\end{proof}

\begin{corollary}
\label{C:chosenfixedpoints}
Let $n = 2r+1\ge 3$ be an odd integer and let $T$ be a set of representatives
for the orbits of $\PGL_2(k)$ acting on the effective divisors of $\Pover{k}$ of
degree at most~$r+1$. Suppose $f$ is a monic irreducible homogenous polynomial
of degree~$n$. Then there is an $f'$ in the $\PGL_2(k)$ orbit of $f$ such that 
the Frobenius divisor of $f'$ lies in~$T$. \qed
\end{corollary}

\subsection{Rational functions with a given divisor of fixed points}
\label{SS:fixedpoints}

The following proposition tells us how to find all rational functions with a
given divisor of fixed points.

\begin{proposition}
\label{P:fixedpoints}
Let $D$ be an effective divisor on $\PP^1$ of degree $r+1\ge 2$ and let $p$ be
the unique monic homogeneous polynomial of degree $r+1$ with $\Div p = D$.
\begin{itemize}
\item Suppose $\infty$ is not in the support of $D$, so that $p$ is not 
      divisible by $y$. Then the set of degree-$r$ rational functions $F$ with 
      $\Phi(F) = D$ is equal to the set 
      \[ 
      \Bigl\{\frac{xh-p}{yh}\Bigr\}\,,
      \] 
      where $h$ ranges over the set of monic polynomials in $R_{r}$ that are 
      coprime to $yp$.
\item Suppose $\infty$ is in the support of $D$, so that $p$ is a multiple 
      of~$y$. Then the set of degree-$r$ rational functions $F$ with 
      $\Phi(F) = D$ is equal to the set 
      \[ 
      \Bigl\{\frac{xh-cp}{yh}\Bigr\}\,,
      \] 
      where $h$ ranges over the set of monic polynomials in $R_{r}$ whose GCDs 
      with $p$ are equal to $y$, and where $c$ ranges over all elements of 
      $\kstar$ with $cp\not\equiv xh\bmod y^2$.
\end{itemize}
\end{proposition}

\begin{proof}
In order for a rational function $F$ to have $\Phi(F) = D$, we must be able to 
write $F = g/h$ for two coprime elements $g,h\in R_{r}$, with $h$ monic, so that
$\Div(xh-yg) = D$. This holds if and only if $xh-yg = cp$ for some $c\in\kstar$.

Suppose we have $xh-yg = cp$ for such a $g$, $h$, and~$c$, and suppose $p$ is
not divisible by~$y$. Then $h$ must also be coprime to $y$. Since $p$ and $h$ 
are both monic, it follows that the coefficient of $x^{r+1}$ in $p$, and the
coefficient of $x^r$ in $h$, are both equal to~$1$. Since the monomial $x^{r+1}$
does not occur in $yg$, it follows that $c=1$, so $g = (xh-p)/y$. Finally, we
see that $h$ and $p$ can share no common factor other than $y$, because
otherwise $h$ and $g$ would not be coprime. Conversely, if $h$ is a monic
polynomial of degree~$r$ that is coprime to~$yp$, then $(xh-p)/(yh)$ is a 
degree-$r$ function with divisor of fixed points equal to~$D$.

Now suppose $xh-yg = cp$ for such a $g$, $h$, and~$c$, where $p$ is divisible
by~$y$. Then $h$ must also be divisible by $y$. But $p$ and $h$ cannot both be 
divisibly by $y^2$, for otherwise $g$ would be divisible by~$y$, contradicting 
the requirement that $g$ and $h$ be coprime. Then $g = x(h/y) - c(p/y)$, and if
$h$ and $p$ had any common factors other than $y$, this common factor would 
divide $g$ as well. Also, since $g$ must not be divisible by $y$, we see that 
$cp\not\equiv xh\bmod y^2$. Conversely, if $h$ is a monic polynomial of degree
$r$ with $(h,p)=y$ and if $c\in\kstar$ satisfies $cp\not\equiv xh\bmod y^2$, 
then $(xh-cp)/(yh)$ is a degree-$r$ function with divisor of fixed points equal
to~$D$.
\end{proof}

\begin{corollary}
\label{C:Ffromfixedpoints}
Let $D$ be an effective divisor on $\PP^1$ of degree~$r+1$. The set of 
degree-$r$ rational functions $F$ with $\Phi(F) = D$ contains at most $q^{r}$
elements, where $q = \#k$, and it can be enumerated in time $O(q^r)$, measured
in arithmetic operations in $k$.\qed
\end{corollary}

\subsection{Places with a given Frobenius function}
\label{SS:Frobeniusfunctioninversion}

In this section we show how to find all the places of odd degree $n\ge 3$ with a
given Frobenius function. First we consider Frobenius functions of degree
greater than~$1$.

\begin{theorem}
\label{T:FFinversion}
Let $F$ be a rational function on $\Pover{\FF_q}$ of degree $r>1$ and let $n$ be
an odd integer integer with $n\ge 2r + 1$. Let $G\colonequals F^{(n)}$ be the
$n$th iterate of~$F$. Then every degree-$n$ place with Frobenius function equal 
to $F$ occurs in the divisor of fixed points of~$G$.
\end{theorem}

\begin{proof}
Let $P$ be a degree-$n$ place whose Frobenius function is~$F$ and let $f$ be the
monic homogeneous polynomial whose divisor is~$P$. Let $\alpha$ be a root 
of~$f$, so that we have $\proj{\alpha^q}{1} = F(\proj{\alpha}{1})$. Applying
Frobenius to both sides, we find that
\[
\proj{\alpha^{q^2}}{1} = F(\proj{\alpha^q}{1})= F^{(2)}(\proj{\alpha}{1})\,,
\]
where the first equality follows from the fact that $F$ is $\FF_q$-rational.
Continuing in this way, we see that 
\[
\proj{\alpha}{1} = \proj{\alpha^{q^n}}{1} = F^{(n)}(\proj{\alpha}{1})\,.
\]
It follows that $\alpha$ is a root of the function defining the divisor of fixed
points for~$G$, so the place $P$ appears in $\Phi(G)$.
\end{proof}

Theorem~\ref{T:FFinversion} gives us a simple algorithm to find the degree-$n$
places with a given Frobenius function $F$ of degree at least~$2$: We go through
the finite set of degree-$n$ places $P$ appearing in $\Phi(G)$, where 
$G\colonequals F^{(n)}$, and for each such $P$ we check to see whether $F$ is
the Frobenius function for~$P$. The degree of $G$ is large, but bounded in terms
of~$n$, so for fixed $n$ this algorithm takes time polynomial in $\log q$, with
the primary step being to find all degree-$n$ factors of the polynomial defining
the divisor of fixed points of~$G$.

But Theorem~\ref{T:FFinversion} only deals with rational functions $F$ of degree
greater than~$1$. The reason for this is that if $F$ is a Frobenius function of
degree~$1$, then the rational function $G$ in the theorem will be the identity
function, and its divisor of fixed points is not defined. To find places with
Frobenius functions of degree~$1$, we instead use the following result.

\begin{theorem}
\label{T:FF1}
If there are places of $\Pover{\FF_q}$ of odd degree $n>1$ with Frobenius
functions of degree~$1$, then either $n$ is the prime divisor of $q$, or $n$
divides $q-1$, or $n$ divides~$q+1$. 

Suppose $n>1$ is an odd integer that is either the prime divisor of $q$ or a
divisor of $q-1$ or $q+1$. Let $s$ be an element of $\FF_{q^2}\setminus\FF_q$.
If $n$ is the prime divisor of $q$, choose an element $t$ of $\FF_q$ with 
nonzero absolute trace and let $T$ be the singleton set 
$\{ x^n - xy^{n-1} - ty^n \}$. If $n$ divides $q-1$, let $\zeta$ be a generator
of $\FF_q^\times$ and let $T$ be the set
\[
\{ x^n - \zeta^i y^n \ \vert \ 0 < i<n/2, \ \gcd(i,n) = 1 \}\,.
\]
If $n$ divides $q+1$, let $\xi\in\FF_{q^{2n}}$ be an element of order $n(q+1)$
and let $T$ be the set of \textup(homogenized\textup) minimal polynomials of the 
elements
\[
\biggl\{\frac{s \xi^i + s^q}{\xi^i + 1} \ \Big\vert \  0< i<n/2, \ \gcd(i,n) = 1 \biggr\}\,,
\]
where $s$ is as above. Then in every case, the divisors of the elements of $T$
form a complete set of unique representatives for the action of $\PGL_2(\FF_q)$
on the set of degree-$n$ places of $\Pover{\FF_q}$ with Frobenius functions of
degree~$1$.
\end{theorem}

\begin{proof}
Suppose $P$ is a place of $\Pover{\FF_q}$ of odd degree $n>1$ whose Frobenius
function $F$ has degree~$1$. Since $F$ has degree~$1$ we can view it as an
element of $\PGL_2(\FF_q)$. Let $f$ be the monic irreducible homogeneous
polynomial in $\FF_q[x,y]$ that defines~$P$, and let $\alpha\in\Fbar_q$ be a
root of~$f$. Then for every $i\ge 0$ we have
\[
\proj{\alpha^{q^i}}{1} = F^{(i)}(\proj{\alpha}{1})\,,
\]
so $F^{(i)}$ is not the identity for $0<i<n$. On the other hand, $F^{(n)}$ fixes
$\proj{\alpha}{1}$ and all of its conjugates, so it must be the identity. In
other words, as an element of $\PGL_2(\FF_q)$, the function $F$ has order~$n$.

An easy exercise shows that the set of orders of elements of $\PGL_2(\FF_q)$
consists of the union of the prime divisor of~$q$, the divisors of $q-1$, and 
the divisors of $q+1$. This proves the first statement of the theorem.

Let $P_\infty$ be the place of $\Pover{\FF_q}$ at infinity, let $P_0$ be the
place at $0$, and let $P_s$ be the degree-$2$ place whose geometric points are
$\proj{s}{1}$ and $\proj{s^q}{1}$, where $s$ is the element of 
$\FF_{q^2}\setminus\FF_q$ from the statement of the theorem. If the Frobenius
function of a place $P$ has degree~$1$, then the Frobenius divisor of $P$ has
degree~$2$. Every divisor of degree~$2$ on $\Pover{\FF_q}$ is in the 
$\PGL_2(\FF_q)$ orbit of exactly one of the following divisors: 
$D_1\colonequals 2P_\infty$, $D_2\colonequals P_\infty + P_0$, and 
$D_3\colonequals P_s$. Therefore, by Lemma~\ref{L:PGLequivariance} we may choose
our orbit representatives for the action of $\PGL_2(\FF_q)$ on the degree-$n$
places with Frobenius functions of degree~$1$ to have one of these three
divisors as their Frobenius divisors.

Consider the set $S_1$ of monic irreducible degree-$n$ homogenous polynomials 
$f$ that define places whose Frobenius divisors are equal to $D_1$. By 
Proposition~\ref{P:fixedpoints}, the functions $F$ with $\Phi(F) = D_1$ are the
functions $x/y + r$ for $r\in\FF_q$ with $r\ne 0.$ Every such function $F$ has
order equal to the prime divisor of~$q$, so in this case $n$ must be that prime.

The subgroup of $\PGL_2(\FF_q)$ that stabilizes the set $S_1$ is the ``$ax+b$'' 
group --- that is, the classes of the upper triangular matrices. We check that 
if $\alpha$ and $\beta$ are elements of $\FF_{q^n}$ such that both 
$\alpha^q - \alpha$ and $\beta^q - \beta$ are nonzero elements of~$\FF_q$, then 
there is an element of this subgroup that takes $\alpha$ to $\beta$. Thus, all
of the places with Frobenius divisors equal to $D_1$ are equivalent up to 
$\PGL_2(\FF_q)$, and one of them is given by the divisor of the unique 
polynomial in the set $T$ given in the theorem in the case where $n$ is the
prime divisor of~$q$.

Now consider the set $S_2$ of monic irreducible degree-$n$ homogenous 
polynomials $f$ that define places whose Frobenius divisors are equal to~$D_2$. 
By Proposition~\ref{P:fixedpoints}, the functions $F$ with $\Phi(F) = D_2$ are
the functions $rx/y$ for $r\in\FF_q$ with $r\ne 0$ and $r\ne 1$, so if $f$ is an
element of $S_2$ and if $\alpha\in\FF_{q^n}$ is a root of $f$, then 
$\alpha^q = r\alpha$ for some $r\in\FF_q^\times$. Since $F\in\PGL_2(\FF_q)$ has 
order~$n$, the element $r$ of $\FF_q^\times$ must also have order~$n$, so 
$n\mid q-1$. Conversely, if $\alpha$ is an element of $\FF_{q^n}^\times$ such
that $\alpha^{q-1}$ has order~$n$, then $\alpha$ defines a place with Frobenius
divisor $D_2$. Since $n$ divides $q-1$, there are elements of $\FF_{q^n}^\times$
of order $n(q-1)$; let $\omega$ be one of these. Then the $\alpha$ that give
places of degree~$n$ with Frobenius divisor $D_1$ are exactly the elements 
$\omega^i$ with $\gcd(i,n) = 1$.

The subgroup of $\PGL_2(\FF_q)$ that stabilizes the set $S_2$ consists of the
classes of the diagonal and antidiagonal matrices. The action of this subgroup
on the set $\{\omega^i \ \vert\ \gcd(i,n) = 1\}$ is generated by multiplication
by nonzero elements of $\FF_q$ (which are exactly the powers of $\omega^n$) and
by inversion. It is clear that the set 
${\{\omega^i \ \vert \ 0 <i<n/2, \ \gcd(i,n) = 1\}}$ is a complete set of unique
representatives for this action. The homogenized minimal polynomials of these
elements are exactly the elements of the set $T$ given in the theorem in the
case $n\mid q-1$.

Finally we consider the set $S_3$ of monic irreducible degree-$n$ homogenous
polynomials $f$ that define places whose Frobenius divisors are equal to
$D_3 = P_s$. Let $p$ be the monic homogeneous polynomial that defines the place 
$P_s$, and write $p=x^2 + axy + by^2$ for $a,b \in \FF_q$. Now 
Proposition~\ref{P:fixedpoints} tells us that the functions $F$ with 
$\Phi(F) = D_3$ are the functions of the form 
$F_r\colonequals ((r-a)x - by)/(x + ry)$ for $r\in\FF_q$.

Given an $\alpha\in\FF_{q^n}$ with $\alpha^q = ((r-a)\alpha - b)/(\alpha + r)$,
set $\beta = (s^q-\alpha)/(\alpha-s)$, so that $\beta\in\FF_{q^{2n}}$ and
$\alpha = (s\beta + s^q)/(\beta + 1)$. Note that then 
\[
\beta^q = \Bigl(\frac{r+s^q}{r+s\phantom{^q}}\Bigr) \frac{1}{\beta}\,,
\]
and more generally
\begin{equation}
\label{EQ:betapowers}
\beta^{q^i} = 
\begin{cases}
\displaystyle \Bigl(\frac{r+s\phantom{^q}}{r+s^q}\Bigr)^i \beta      & \text{if $i$ is even\textup{;}}\\[2.5ex]
\displaystyle \Bigl(\frac{r+s^q}{r+s\phantom{^q}}\Bigr)^i \beta^{-1} & \text{if $i$ is odd.}\\
\end{cases}
\end{equation}
Since $\alpha^{q^n} = \alpha$ and $n$ is odd, we find that
\[
\frac{s\beta + s^q}{\beta + 1} 
  = \alpha
  = \alpha^{q^n}  
  = \frac{s^q\beta^{q^n} + s}{\beta^{q^n} + 1}
  = \frac{s\beta^{-q^n} + s^q}{\beta^{-q^n} + 1}\,,
\]
so that $\beta^{-q^n} = \beta$. Then~\eqref{EQ:betapowers} shows that 
$c\colonequals(r+s^q)/(r+s)\in\FF_{q^2}$ has order~$n$. It is also clear that
the norm of $c$ from $\FF_{q^2}$ to $\FF_q$ is~$1$, so its order divides~$q+1$,
and we see that $n\mid q+1$. Since $\beta^{q+1} = c$, we see that the $\beta$
has order dividing $n(q+1)$, so that $\beta$ is a power of the element 
$\xi\in\FF_{q^{2n}}$ chosen in the statement of the theorem. Furthermore, if we
write $\beta = \xi^i$, then we must have $\gcd(i,n) = 1$, for otherwise 
$\beta^{q+1}$ would not have order~$n$. On the other hand, we check that if we 
set $\alpha = (s\xi^i + s^q)/(\xi^i + 1)$ for an integer $i$ with 
$\gcd(i,n) = 1$, then we have
\[
\alpha^q = \frac{(r-a)\alpha - b}{x + r}
\]
where 
\[
r = \frac{-\xi^{q+1}s + s^q}{\xi^{q+1}-1}\,.
\]
Using the fact that $\xi$ has order $n(q+1)$ and that $n\mid q+1$ we see that 
$\xi^{q^2+q}\cdot \xi^{q+1} = 1$, and from this we calculate that $r^q = r$, so
that $r\in \FF_q$. Thus, the place associated to this $\alpha$ has Frobenius
divisor equal to~$D_3$.

We calculate that the subgroup of $\PGL_2(\FF_q)$ that stabilizes the set $S_3$ 
is generated by the classes of the matrices 
$M_u\colonequals \twobytwo{u-a}{-b}{1}{\phantom{-}u}$ (which fix $s$ and $s^q$) 
and the matrix $N\colonequals \twobytwo{-1}{-a}{\phantom{-}0}{\phantom{-}1}$
(which swaps $s$ and $s^q$).

Suppose $\alpha$ and $\beta$ are as above, set 
$\alpha' \colonequals M_u(\alpha)$, and let 
$\beta' = (s^q-\alpha')/(\alpha'-s)$. We check that 
$\beta' = (u+s)/(u+s^q)\beta.$ Since $(u+s)/(u+s^q)$ is an element of
$\FF_{q^2}$ whose norm to $\FF_q$ is~$1$, we see that it is a power of~$\xi^n$. 
Therefore, if we write $\beta = \xi^i$ and $\beta' = \xi^{i'}$, then 
$i' \equiv i \bmod n$. Likewise, if we let $\alpha' = N(\alpha)$ and
$\beta' = (s^q-\alpha')/(\alpha'-s)$, then $\beta' = 1/\beta$, so if we write 
$\beta = \xi^i$ and $\beta' = \xi^{i'}$, then $i' \equiv -i \bmod n(q+1)$. From
this we see that the set $\{\xi^i \ \vert \ 0<i<n/2, \ \gcd(i,n) = 1\}$ is a
complete set of unique representatives for the action of the stabilizer of $S_3$
on the set of~$\beta$s, so the set 
\[
\biggl\{\frac{s \xi^i + s^q}{\xi^i + 1} \ \Big\vert \  0< i<n/2, \ \gcd(i,n) = 1 \biggr\}\,.
\]
is a complete set of unique representatives for the action of the stabilizer of 
$S_3$ on the set of $\alpha$s whose associated places have Frobenius 
divisor~$D_3$. The theorem follows.
\end{proof}

\begin{remark}
The Frobenius function of a degree-$3$ place has degree~$1$, so 
Theorem~\ref{T:FF1} gives us a complete set of unique representatives for the
action of $\PGL_2(\FF_q)$ on the degree-$3$ places. Of course, we already know
that the action of  $\PGL_2(\FF_q)$ on these places is transitive, so there is
only one element in such a set of representatives. It is interesting to see how
the representative provided by the theorem depends on the residue class of $q$
modulo~$3$, and to see how the shape of the Frobenius divisor also depends on
this residue class.
\end{remark}

\subsection{Cross polynomials}
\label{SS:crosspolynomials}

In this subsection we review the definition and main property of ``cross
polynomials'' from~\cite{Howe2024}.

\begin{definition}[Definition~4.1 from \cite{Howe2024}]
\label{D:cross}
Let $f$ be a monic irreducible homogeneous polynomial in $R_n$ with $n\ge 4$, 
let $\alpha\in\FF_{q^n}$ be a root of~$f$, and let $\chi\in\FF_{q^n}$ be the 
cross ratio of $\alpha$, $\alpha^q$, $\alpha^{q^2}$, and $\alpha^{q^3}$; that
is,
\[
\chi\colonequals
\frac{(\alpha^{q^3}-\alpha^q)(\alpha^{q^2}-\alpha)}
     {(\alpha^{q^3}-\alpha)(\alpha^{q^2}-\alpha^{q})}.
\]
The \emph{cross polynomial} $\Cross(f)$ of $f$ is the characteristic polynomial
of $\chi$ over~$\FF_q$. If $P$ is a place of $\Pover{\FF_q}$, we let $\Cross(P)$
be the cross polynomial of the monic irreducible polynomial whose divisor
is~$P$.
\end{definition}

\begin{theorem}[Theorem~4.2 from \cite{Howe2024}]
\label{T:cross}
Two monic irreducible polynomials in $R_n$ with $n\ge 4$ lie in the same orbit 
under the action of $\PGL_2(\FF_q)$ if and only if they have the same cross
polynomial. \qed
\end{theorem}

\subsection{Orbit labels and orbit representatives for 
\texorpdfstring{$\PGL_2(\FF_q)$}{PGL2(Fq)} in 
\texorpdfstring{$\PGL_2(\FF_{q^2})$}{PGL2(Fq2)}}
\label{SS:cosetreps}

In our algorithm, we will need to have an explicit list of (right) coset
representatives for $\PGL_2(\FF_q)$ in $\PGL_2(\FF_{q^2})$, as well as an
easy-to-compute invariant that determines whether two elements of 
$\PGL_2(\FF_{q^2})$ lie in the same coset of $\PGL_2(\FF_q)$. We begin with the
explicit list.

An element of $\PGL_2(\FF_{q^2})$ is determined by where it sends $\infty$, $0$, 
and $1$, and given any three distinct elements of $\PP^1(\FF_{q^2})$, there is
an element of $\PGL_2(\FF_{q^2})$ that sends $\infty$, $0$, and $1$ to those
three elements. Thus, as in~\cite{Howe2024}, we may represent elements of
$\PGL_2(\FF_{q^2})$ by triples $(\zeta,\eta,\theta)$ of pairwise distinct
elements of $\PP^1(\FF_{q^2})$, indicating the images of $\infty$, $0$, and~$1$.

\begin{proposition}[Proposition~5.1 from~\cite{Howe2024}]
\label{P:CosetRepsPGL2}
Let $q$ be a prime power, let $\omega$ be an element of 
$\FF_{q^2}\setminus\FF_q$ and let $\gamma$ be a generator for the 
multiplicative group of $\FF_{q^2}$. Let $B$ be the set 
$\{(\omega\gamma^i + \omega^q)/(\gamma^i+1) \ \vert \ 0\le i < q-1\bigr\}.$ 
The following elements give a complete set of unique coset representatives for
the left action of $\PGL_2(\FF_q)$ on $\PGL_2(\FF_{q^2})$\textup{:}
\begin{enumerate}[label=\textup{\arabic*}.,ref=\arabic*]
\item \label{CRPGL21}
      $(\infty, 0, 1)$\textup{;}
\item \label{CRPGL22}
      $\bigl\{ (\infty, 0, \omega + a) \ \vert \ a \in \FF_q \bigr\}$\textup{;}
\item \label{CRPGL23}
      $\bigl\{ (\infty, \omega, \theta) \ \vert \ \theta \in \FF_{q^2} \text{\ with\ } \theta \neq \omega \bigr\}$\textup{;}
\item \label{CRPGL24}
      $\bigl\{ (\omega, \omega^q, \theta) \ \vert \ \theta\in B \bigr\}$\textup{;}
\item \label{CRPGL25}
      $\bigl\{ (\omega, \eta, \theta) \ \vert \ \eta\in B,  \theta \in \PP^1(\FF_{q^2})
      \text{\ with\ } \theta \neq \omega \text{\ and\ }\theta\ne \eta \bigr\}$.\qed
\end{enumerate}
\end{proposition}

The proof of Proposition~\ref{P:CosetRepsPGL2} given in~\cite[\S5]{Howe2024}
shows how we can determine the orbit representative attached to a given element
$\Gamma$ of $\PGL_2(\FF_{q^2})$, and we encourage the reader to consult it.
Unfortunately, determining this representative involves computing a discrete 
logarithm. For our main algorithm we are not allowing ourselves enough memory to
keep a table of logarithms, and we will be needing to identify orbits often
enough that computing discrete logarithms directly would take too much time.
Therefore, we introduce an orbit invariant 
$\CL\colon \PGL_2(\FF_{q^2})\to \FF_{q^2}\times\FF_{q^2}\times\FF_{q^2}$ that is
easier to compute.

\begin{definition}
\label{D:orbitlabels}
Let $q$ be a prime power, let $\omega$ be an element of 
$\FF_{q^2}\setminus\FF_q$ and let $\gamma$ be a generator for the 
multiplicative group of $\FF_{q^2}$. Given $\Gamma\in \PGL_2(\FF_{q^2})$, 
represented by a triple $(\zeta_0,\eta_0,\theta_0)$ of elements of 
$\PP^1(\FF_{q^2})$, we define $\CL(\Gamma)$ as follows:
\begin{enumerate}
\item If $\zeta_0$, $\eta_0$, and $\theta_0$ are all elements of $\PP^1(\FF_q)$, we
      define $\CL(\Gamma)\colonequals (0,0,1)$.
\item If $\zeta_0$ and $\eta_0$ are elements of $\PP^1(\FF_q)$ but $\theta_0$ is
      not, we compute the orbit representative $(\infty,0,\omega+a)$ of $\Gamma$
      using the procedure in~\cite[\S5]{Howe2024} and set 
      $\CL(\Gamma)\colonequals (0,0,\omega+a)$.
\item If $\zeta_0$ lies in $\PP^1(\FF_q)$ but $\eta_0$ does not, we compute the
      orbit representative $(\infty,\omega,\theta)$ of $\Gamma$ using the
      procedure in~\cite[\S5]{Howe2024} and set 
      $\CL(\Gamma)\colonequals (0,\omega,\theta)$.
\item If $\zeta_0$ does not lie in $\PP^1(\FF_q)$, we use elements of 
      $\PGL_2(\FF_q)$ to move $\Gamma$ to an element of the form 
      $(\omega,\eta,\theta)$ using the procedure in~\cite[\S5]{Howe2024}.
      \begin{enumerate}
      \item If $\eta = \omega^q$ then we set
            $\CL(\Gamma)\colonequals (1,0,\Norm_{\FF_{q^2}/\FF_q}(\Phi(\theta)))$,
            where 
            $\Phi\colonequals \twobytwo{-1}{\phantom{-}\omega^q}{\phantom{-}1}{-\omega\phantom{^q}}$
            is as in~\cite[\S5]{Howe2024}. Note that if $g$ is an element of
            $\PGL_2(\FF_q)$ that fixes $\omega$ and $\omega^q$, then 
            $\Phi(g\theta)$ differs from $\Phi(\theta)$ by a multiplicative 
            factor whose norm to $\FF_q$ is equal to~$1$, so $\CL(\Gamma)$ is 
            well defined.
      \item If $\eta \ne \omega^q$ then we set
            $\CL(\Gamma)\colonequals (1,\Norm_{\FF_{q^2}/\FF_q}(\Phi(\eta)),\Phi(\theta)/\Phi(\eta))$.
            Again we note that if we replace $\eta$ and $\theta$ by $g\eta$ and
            $g\theta$ for an element $g$ of $\PGL_2(\FF_q)$ that fixes $\omega$,
            then neither the second nor third entries of the triple 
            $\CL(\Gamma)$ given above will change.
      \end{enumerate}
\end{enumerate}
We call the value $\CL(\Gamma)$ the \emph{orbit label} of $\Gamma$.
\end{definition}

\begin{proposition}
\label{P:orbitlabel}
Two elements of $\PGL_2(\FF_{q^2})$ lie in the same $\PGL_2(\FF_q)$ orbit if and
only if their orbit labels are equal.
\end{proposition}

\begin{proof}
The proof of \cite[Proposition~5.1]{Howe2024} and the discussion above show that
that $\CL(\Gamma)$ depends only on the orbit of~$\Gamma$. On the other hand, 
direct computation shows that the orbit representatives given in 
Proposition~\ref{P:CosetRepsPGL2} have different orbit labels.
\end{proof}

\subsection{Orbit labels and orbit representatives for 
\texorpdfstring{$\PGL_2$}{PGL2} acting on pairs of distinct
degree-\texorpdfstring{$2$}{2} places}
\label{SS:orbits22}

Analogously to the situation in the preceding section, for our main algorithm we
will need to have orbit representatives and orbit labels for the action of
$\PGL_2(\FF_q)$ on degree-$4$ effective divisors of $\Pover{\FF_q}$ that are the
sum $P_1+P_2$ of two distinct places of degree~$2$. Such divisors correspond to
homogenous quartics $f\in k[x,y]$ that factor into the product $f_1 f_2$ of two
distinct irreducible monic quadratics, with $\Div f_i = P_i$. Orbit 
representatives for the action of $\PGL_2(\FF_q)$ on such quartics are provided 
by~\cite[Theorems~3.4 and~3.8]{Howe2024}, so our goal here is simply to provide
an easily-computable orbit label.

Given a quartic $f$ that factors into the product $f_1 f_2$ of two distinct 
irreducible monic homogeneous quadratics, we write $f_1 = x^2 + sxy + ty^2$ and 
$f_2 = x^2 + uxy + vy^2$, and we define $\lambda(f)$ by the formula
\[
\lambda(f)\colonequals \frac{(s-u)(sv-tu) + (t-v)^2}{(s^2-4t)(u^2-4v)}\,.
\]
If $P_1$ and $P_2$ are the divisors of $\Pover{\FF_q}$ corresponding to $f_1$ 
and $f_2$, we also set $\lambda(P_1+P_2) \colonequals \lambda(f)$.

\begin{lemma}
The function $\lambda$ is constant on $\PGL_2(\FF_q)$ orbits of homogeneous
quartics that factor as the product of distinct monic irreducible quadratics,
and it takes different values on different orbits.
\end{lemma}

\begin{proof}
We note that in odd characteristic we have $1 + 4\lambda(f) = \mu(f)$, where 
$\mu$ is the function defined in Remark~3.5 of~\cite{Howe2024}. As is observed 
there, the function $\mu$ is constant on $\PGL_2(\FF_q)$ orbits of homogeneous
quartics of the type we are considering, and it takes different values on 
different orbits. Therefore the same is true of $\lambda$ when $q$ is odd.

We let the reader verify that the function $\lambda$ is constant on 
$\PGL_2(\FF_q)$ orbits of such quartics even in characteristic~$2$. On the other
hand, it is easy to verify that $\lambda$ takes distinct values on the orbit 
representatives for such quartics provided by~\cite[Theorem~3.8]{Howe2024}.
Thus, the lemma also holds when $q$ is even.
\end{proof}

\section{A note about recursion}
\label{S:algorithm}

All of the main algorithms in this paper are defined recursively: The algorithm
to enumerate places of degree $n$ up to $\PGL_2$ uses the algorithm to enumerate
effective divisors of degree at most~$(n+1)/2$ up to $\PGL_2$, and the algorithm
to enumerate effective divisors of degree $n$ up to $\PGL_2$ uses the algorithms
to enumerate places of degree at most $n$ up to $\PGL_2$. Since all of these 
algorithms are designed to take space $O(\log q)$, there is no room for saving a 
list of all the divisors or places of a given degree, up to $\PGL_2$. So we have
to be specific about what we mean when, as a step in one algorithm, we say 
something like ``for every $\PGL_2$ orbit of effective divisors of degree~$d$,
do\ldots''

One easy-to-conceptualize solution to this issue is to note that every algorithm
that takes time $\Otilde(q^e)$ and space $O(\log q)$ to produce a list of
divisors --- let's call it \texttt{CompleteList(q)} --- can be converted into an
algorithm \texttt{NextElement(D,data)} that takes as input one list element $D$
(plus some auxiliary data) and outputs the next list element (plus some 
auxiliary data), and that also uses $O(\log q)$ space and a total time of 
$\Otilde(q^e)$ to run through the whole list one elements at a time.

The idea is simple: The steps in \texttt{NextElement(D,data)} are the same as in
\texttt{CompleteList(q)}, except that where \texttt{CompleteList(q)} would
output a list element~$D$, the algorithm \texttt{NextElement(D,data)} outputs
$D$ together with all the details of the algorithm's internal state --- the line
of code it is at, together with the values of all of the internal variables --- 
and then stops. Since \texttt{CompleteList(q)} is assumed to take space
$O(\log q)$, the amount of internal state is bounded by the same amount. When 
\texttt{NextElement(D,data)} is called again, with $D$ and its associated data
as input, it simply continues as \texttt{CompleteList(q)} would, until it is
time to output another list element. 

This means that in our algorithms we can include ``for'' statements like the 
example given at the beginning of this section, and we will do so without
further comment.

\section{The structure of our argument}
\label{S:structure}

Our algorithm is recursive, and the proof that it runs correctly and within the
stated bounds on time and space is inductive. Since the various steps are
described in different sections, it will be helpful to outline the structure of
the induction here.

Theorems~1.1 and 3.8 of \cite{Howe2024} show that we can enumerate a complete
set of unique representatives for the action of $\PGL_2(\FF_q)$ on the 
degree-$4$ places of $\Pover{\FF_q}$ in time $O(q)$ and space $O(\log q)$. The
$\PGL_2(\FF_q)$ orbits of places of degree less than $4$ can be enumerated in 
time $O(\log q)$ and space $O(\log q)$, as can the $\PGL_2(\FF_q)$ orbits of 
effective divisors of degree less than~$4$. These results form the base case for
our induction.

In Section~\ref{S:divisors} we show that if, for every $d$ with $3\le d\le n$, 
we can enumerate a complete set of unique representatives for the action of
$\PGL_2(k)$ on places of degree~$d$ in time $\Otilde(q^{d-3})$ and space
$O(\log q)$, then we can enumerate a complete set of unique representatives for
the action of $\PGL_2(k)$ on the effective divisors of degree $n$ in time
$\Otilde(q^{n-3})$ and space $O(\log q)$. Note that this is enough to prove that
Corollary~\ref{C:divisors} follows from Theorem~\ref{T:places}.

For our induction, we suppose that we can enumerate a complete set of unique 
representatives for the action of $\PGL_2(\FF_q)$ on places of degree $d$ in 
time $\Otilde(q^{d-3})$ and space $O(\log q)$ for every $d$ with 
$3\le d\le n-1$. By the argument in Section~\ref{S:divisors}, this means that
for every such $d$ we can also enumerate a complete set of unique 
representatives for the action of $\PGL_2(\FF_q)$ on effective divisors of
degree $d$ in time $\Otilde(q^{d-3})$ and space $O(\log q)$.

In Section~\ref{S:simpleodd} we show that if $n$ is odd, then this induction
hypothesis shows that we can enumerate orbit representatives for the action of 
$\PGL_2(\FF_q)$ on places of degree $n$ in time $\Otilde(q^{n-3})$ and space 
$O(\log q)$. In Section~\ref{S:simpleeven} we prove the same thing for even~$n$.
Together, the arguments in these sections provide a proof of 
Theorem~\ref{T:places}.

\section{From \texorpdfstring{$\PGL_2$}{PGL2} orbits of places to 
\texorpdfstring{$\PGL_2$}{PGL2} orbits of effective divisors}
\label{S:divisors}

In this section we assume that for every $d$ with $3\le d\le n$, we have an
algorithm that can enumerate orbit representatives for $\PGL_2(k)$ acting on the
degree-$d$ places of $\Pover{k}$ in time $\Otilde(q^{d-3})$ and using 
$O(\log q)$ space, where $q=\#k$. We show that under this assumption, we can 
produce an algorithm that can enumerate orbit representatives for $\PGL_2(k)$ 
acting on the effective degree-$n$ divisors of $\Pover{k}$ in time
$\Otilde(q^{n-3})$ and using $O(\log q)$ space.

Let $D$ be an effective divisor of $\Pover{k}$ and write $D = \sum m_P P$, where
the sum is over the places $\Pover{k}$, the coefficients $m_P$ are nonnegative 
integers, and all but finitely many of the $m_P$ are zero. The \emph{support} of
$D$ is the divisor 
\[
\Supp D \colonequals \sum_{P \, : \, m_P\ne 0} P.
\]
The divisor $D$ is \emph{reduced} if it is equal to its own support. We begin by
showing how to enumerate reduced divisors.

\subsection{Enumerating reduced effective divisors 
up to \texorpdfstring{$\PGL_2$}{PGL2}}
\label{SS:reduced}

Suppose $D$ is a reduced effective divisor of degree at most $d$, so that
$D = P_1 + \cdots P_r$ for distinct places $P_i$, and so that if we let $m_i$
be the degree of $P_i$, then $m_1 + \cdots + m_r \le d$. 
Following~\cite[\S 7]{Howe2024}, we say that sequence $(m_i)_i$, listed in
non-increasing order, is the \emph{Galois type} of $D$. The \emph{degree} of a
Galois type is the sum of its elements. We will enumerate orbit representatives
for the action of $\PGL_2(\FF_q)$ on reduced effective divisors of degree at
most $d$ by enumerating each Galois type separately.

We have three strategies that together cover all but five Galois types. The
first strategy works with Galois types $(m_i)_i$ for which $m_1\ge 3$. The
second strategy is for Galois types with $m_1 = m_2 = 2$. The third strategy is
for Galois types with $m_1\le 2$, $m_2 = 1$, and with at least three $m_i$ equal
to~$1$. The Galois types not covered by these strategies are $(1)$, $(2)$,
$(1,1)$, $(2,1)$, and $(2,1,1)$. We dispose of these small cases first.

\subsubsection{Small Galois types.} 
Let $q_2$ be an irreducible monic quadratic polynomial in $k[x,y]$. We leave it
to the reader to verify the following:
\begin{itemize}
\item There is only one $\PGL_2(\FF_q)$ orbit of the divisors of Galois type
      $(1)$, and it can be represented by the homogeneous polynomial~$y$.
\item There is only one $\PGL_2(\FF_q)$ orbit of the divisors of Galois type
      $(1,1)$, and it can be represented by the homogeneous polynomial~$xy$.
\item There is only one $\PGL_2(\FF_q)$ orbit of the divisors of Galois type
      $(2)$, and it can be represented by the homogeneous polynomial~$q_2$.
\item There is only one $\PGL_2(\FF_q)$ orbit of the divisors of Galois type
      $(2,1)$, and it can be represented by the homogeneous polynomial~$yq_2$.
\item A complete set of unique representatives for the $\PGL_2(\FF_q)$ orbits of
      divisors of Galois type $(2,1,1)$ is provided
      by~\cite[Theorem~3.6]{Howe2024}.
\end{itemize}

\subsubsection{Galois types with $m_1\ge 3$.}
Choose total orderings on the set of polynomials in $\FF_q[x]$ and on the set of
homogeneous polynomials in $\FF_q[x,y]$. Having chosen these orderings, it makes
sense to speak of one element of one of these sets being ``smaller'' than 
another element of the same set. The following algorithm makes use of cross 
polynomials, as discussed in Section~\ref{SS:crosspolynomials}.

\begin{algorithm}
\label{A:3andup}
Orbit representatives for $\PGL_2(\FF_q)$ acting on the reduced effective
divisors of $\Pover{\FF_q}$ of Galois type $M= (m_1,\ldots,m_r)$, where
$m_1\ge 3$ and the degree of $M$ is at most $n$.
\begin{alglist}
\algin  A prime power $q$ and a Galois type $M = (m_1,\ldots,m_r)$ of degree
        $d\le n$ with $m_1\ge 3$.
\algout A list of monic separable homogeneous polynomials in $\FF_q[x,y]$ of
        degree~$d$ whose associated divisors form a complete set of unique
        representatives for $\PGL_2(\FF_q)$ acting on the divisors of Galois
        type~$M$.
\item Let $M'$ be the Galois type $(m_2,\ldots,m_r)$.        
\item \label{3andup3}
      For every orbit representative $P_1$ for the action of $\PGL_2(\FF_q)$ on
      the degree-$m_1$ places of $\Pover{\FF_q},$ and for every reduced
      effective divisor $D' = P_2 + \cdots + P_r$ of $\Pover{\FF_q}$ of Galois
      type $M'$ that does not contain the place $P_1$, do:
      \begin{algsublist}
      \item Set $\chi\colonequals\Cross(P_1)$.
      \item If $D'$ contains a place $P$ of degree $m_1$ with $\Cross(P)$
            smaller than $\chi$, continue to the next pair $(P_1,D')$.
      \item \label{3andup3c}
            Let $f$ be the monic degree-$d$ homogeneous polynomial representing
            the divisor $D\colonequals P_1 + D'$.
      \item \label{3andup3d}
            Set $F\colonequals\{\Gamma(f)\}$, where $\Gamma$ ranges over the
            elements of $\PGL_2(\FF_q)$ that send some degree-$m_1$ place of $D$
            with cross polynomial $\chi$ to $P_1$.
      \item If $f$ is the smallest element of $F$, output $f$.
      \end{algsublist}
\end{alglist}
\end{algorithm}

\begin{remark}
Note that the $\Gamma$ considered in step~\ref{3andup3}\ref{3andup3d} include
the $\Gamma$ in the $\PGL_2(\FF_q)$ stabilizer of~$P_1$.
\end{remark}

\begin{proposition}
\label{P:3andup}
Algorithm~\textup{\ref{A:3andup}} produces a complete set of unique 
representatives for the orbits of $\PGL_2(\FF_q)$ acting on the reduced
effective divisors of $\Pover{\FF_q}$ of Galois type~$M$. It runs in time 
$\Otilde(q^{d-3})$, measured in arithmetic operations in $\FF_q$, and requires
$O(\log q)$ space.
\end{proposition}

\begin{proof}
The algorithm produces correct results, because every effective divisor $D$ of 
Galois type $M$ is in the same orbit as divisors of the form $P_1 + D'$ where
$P_1$ is in the given set of orbit representatives and $D'$ is a divisor of
Galois type $M'$, none of whose constituent places has degree $m_1$ and has
cross polynomial smaller than that of $P_1$. Every such orbit representative
will be considered in steps~\ref{3andup3}\ref{3andup3c} 
and~\ref{3andup3}\ref{3andup3d}, but only one will satisfy the requirement of
step~\ref{3andup3}\ref{3andup3d} and be output.

The time it takes to process each $(P_1,D')$ pair is polynomial in~$\log q$. The
time to generate the representatives $P_1$ is $\Otilde(q^{m_1-3})$, by the
global hypotheses of this section. The time to generate all the reduced
effective divisors of Galois type $M'$ is $\Otilde(q^{d-m_1})$. All told, the
time required for the algorithm is therefore $\Otilde(q^{d-3})$.

Generating the representatives $P_1$ takes space $O(\log q)$ by hypothesis.
Generating the effective divisors of type $M'$ can also be done in this space 
(keep in mind that $n$ is fixed for the purposes of this algorithm, so any
dependencies on $n$ can be absorbed into constant factors). And the substeps of 
step~\ref{3andup3} can also be accomplished in $O(\log q)$ space. Thus, the
whole algorithm requires only $O(\log q)$ space.
\end{proof}

\subsubsection{Galois types with $m_1 = m_2 = 2$.}
In the preceding subsection we ``absorbed'' all of the action of $\PGL_2(\FF_q)$
into a single place of degree~$m_1$. Now we do something similar with pairs of 
places of degree~$2$. 

\begin{algorithm}
\label{A:22}
Orbit representatives for $\PGL_2(\FF_q)$ acting on the reduced effective
divisors of $\Pover{\FF_q}$ of Galois type $M= (m_1,\ldots,m_r)$, where 
$m_1 = m_2 = 2$ and the degree of $M$ is at most $n$.
\begin{alglist}
\algin  A prime power $q$ and a Galois type $M = (m_1,\ldots,m_r)$ of degree
        $d\le n$ with $m_1 = m_2 = 2$.
\algout A list of monic separable homogeneous polynomials in $\FF_q[x,y]$ of
        degree~$d$ whose associated divisors form a complete set of unique
        representatives for $\PGL_2(\FF_q)$ acting on the divisors of Galois
        type~$M$. 
\item Let $M'$ be the Galois type $(m_3,\ldots,m_r)$.        
\item For every sum $P_1+P_2$ of degree-$2$ places of $\Pover{\FF_q}$ listed in 
      Theorems~3.4 and~3.8 of~\cite{Howe2024}, and for every reduced effective 
      divisor $D' = P_3 + \cdots + P_r$ of $\Pover{\FF_q}$ of Galois type $M'$
      that does not contain either of the places $P_1$ and~$P_2$, do:
      \begin{algsublist}
      \item Set $\chi\colonequals\lambda(P_1+P_2)$, where $\lambda$ is the 
            function defined in \S\ref{SS:orbits22}.
      \item If $P_1 + P_2 + D'$ contains two degree-$2$ places $Q_1$ and $Q_2$
            with $\lambda(Q_1+Q_2)$ smaller than $\chi$, continue to the next
            pair $(P_1+P_2,D')$.
      \item Let $f$ be the monic degree-$d$ homogeneous polynomial representing
            the divisor $D\colonequals P_1 + P_2 + D'$.
      \item Set $F\colonequals\{\Gamma(f)\}$, where $\Gamma$ ranges over the
            elements of $\PGL_2(\FF_q)$ that send some pair $(Q_1,Q_2)$ of
            degree-$2$ places of $D$ with $\lambda(Q_1+Q_2) = \chi$ to 
            $P_1 + P_2$.
      \item If $f$ is the smallest element of $F$, output $f$.
      \end{algsublist}
\end{alglist}
\end{algorithm}

\begin{proposition}
\label{P:22}
Algorithm~\textup{\ref{A:22}} produces a complete set of unique representatives
for the orbits of $\PGL_2(\FF_q)$ acting on the reduced effective divisors of
$\Pover{\FF_q}$ of Galois type~$M$. It runs in time $\Otilde(q^{d-3})$, measured
in arithmetic operations in $\FF_q$, and requires $O(\log q)$ space.
\end{proposition}

\begin{proof}
The proof is entirely analogous to that of Proposition~\ref{P:3andup}.
\end{proof}

\subsubsection{Galois types with $m_2 = 1$ and with at least three $1$s}
This is perhaps the most straightforward case to deal with, since we can
normalize our divisors by demanding that three of their degree-$1$ places be
$\infty$, $0$, and~$1$.

\begin{algorithm}
\label{A:111}
Orbit representatives for $\PGL_2(\FF_q)$ acting on the reduced effective
divisors of $\Pover{\FF_q}$ of Galois type $M= (m_1,\ldots,m_r)$, where 
$m_2 = 1$, at least three of the $m_i$ are equal to~$1$, and the degree of $M$
is at most $n$.
\begin{alglist}
\algin  A prime power $q$ and a Galois type $M = (m_1,\ldots,m_r)$ of degree
        $d\le n$ with $m_2 = 1$ and with at least three $m_i$ equal to~$1$.
\algout A list of monic separable homogeneous polynomials in $\FF_q[x,y]$ of
        degree~$d$ whose associated divisors form a complete set of unique
        representatives for $\PGL_2(\FF_q)$ acting on the divisors of Galois
        type~$M$.
\item Let $M'$ be the Galois type $(m_1,\ldots,m_{r-3})$.
\item Let $D_0$ be the degree-$3$ divisor consisting of the places at $\infty$,
      $0$, and~$1$.
\item For every reduced effective divisor $D' = P_1 + \cdots + P_{r-3}$ of 
      $\Pover{\FF_q}$ of Galois type $M'$ that is disjoint from $D_0$, do:
      \begin{algsublist}
      \item Let $f$ be the monic degree-$d$ homogeneous polynomial representing
            the divisor $D\colonequals D_0 + D'$.
      \item Set $F\colonequals\{\Gamma(f)\}$, where $\Gamma$ ranges over the
            elements of $\PGL_2(\FF_q)$ that send some triple $(Q_1,Q_2,Q_3)$ of
            degree-$1$ places of $D$ to $D_0$.
      \item If $f$ is the smallest element of $F$, output $f$.
      \end{algsublist}
\end{alglist}
\end{algorithm}

\begin{proposition}
\label{P:111}
Algorithm~\textup{\ref{A:111}} produces a complete set of unique representatives
for the orbits of $\PGL_2(\FF_q)$ acting on the reduced effective divisors of
$\Pover{\FF_q}$ of Galois type~$M$. It runs in time $\Otilde(q^{d-3})$, measured
in arithmetic operations in $\FF_q$, and requires $O(\log q)$ space.
\end{proposition}

\begin{proof}
The proof is again essentially the same as that of
Propositions~\ref{P:3andup} and~\ref{P:22}.
\end{proof}

\subsection{Enumerating effective divisors up to \texorpdfstring{$\PGL_2$}{PGL2}}
\label{SS:nonreduced}

Given the results of the preceding subsection, it is a simple matter to produce
a complete set of unique representatives for the action of $\PGL_2(\FF_q)$ on
the degree-$n$ effective divisors of~$\PP^1_{\/\FF_q}$. The support of such a
divisor is a reduced effective divisor of degree at most~$n$, and the algorithms
we have given above will produce a complete set of unique representatives for
these in time $\Otilde(q^{n-3})$ and space $O(\log q)$.

For each divisor $S$ in this set, we do the following. The divisor $S$ is simply
a collection of places $P$ of $\Pover{k}$, and our task consists of two steps:
First, we must find all sets of positive integers $\{m_P\}_{P\in S}$ such that
$\sum_{P\in S} m_P \deg P = n$. Each such set gives a divisor $D$ of degree~$n$.
Then, we must take the list of all such $D$ for the given~$S$, check to see
which $D$ are in the same $\PGL_2(\FF_q)$ orbit, and then choose one element
from each orbit to output.

Finding all the sets $\{m_P\}_{P\in S}$ can be done in constant time, since $n$
is fixed, so we have a bounded number of degree-$n$ divisors associated to a 
given~$S$. The number of pairs of such divisors is also bounded in terms of~$n$.
Comparing two such divisors to see whether they are in the same $\PGL_2(k)$
orbit can be done in time polynomial in $\log q$; the most direct method
involves finding explicit isomorphisms between the residue fields of two places
of the same degree, and we can accomplish this either by using a polynomial
factorization algorithm, or by choosing our data structure for places of degree
$d$ to include an explicit Galois-stable collection of $d$ points in a fixed
copy of $\FF_{q^d}$ (see the following section). In any case, this can be done
in the required time and space.

This sketch shows that we can output a complete set of unique representatives
for the action of $\PGL_2(k)$ on the set of effective divisors of degree $n$ in 
time~$\Otilde(q^{n-3})$ and using $O(\log q)$ space, under the assumption that
we can accomplish the same task for places of degree at most~$n$.

\section{Enumerating places of odd degree up to \texorpdfstring{$\PGL_2$}{PGL2}}
\label{S:simpleodd}

Here we present an algorithm for enumerating representatives for the orbits of
$\PGL_2(k)$ acting on the places of $\PP^1$ of odd degree~$n$ over a finite
field~$k$, using $O(\log q)$ space and $\Otilde(q^{n-3})$ time, where $q = \#k$. 
Our hypothesis throughout this section is that we can accomplish this same task
for places of every degree~$d$ smaller than~$n$, and by the results of the
preceding section this hypothesis implies that we can also enumerate 
representatives for the orbits of $\PGL_2(k)$ acting on the effective
divisors of $\PP^1$ of degree less than~$n$.

The algorithm requires us to manipulate effective divisors of $\Pover{k}$ of
degree up to~$n$, so we need to specify the data structure we use to represent
such divisors. Let $q = \#k$. Before we begin the algorithm proper, we compute
copies of the fields $\FF_{q^i}$ for $1< i\le n$. We represent places of
degree $i>1$ by Galois-stable collections of $i$ elements of $\FF_{q^i}$ that
lie in no proper subfield, we represent places of degree $1$ by elements of
$\FF_q$ together with the symbol $\infty$, and we represent divisors as 
formal sums of such places. This representation makes it easy to check whether 
two divisors of the same degree are in the same orbit under the action of
$\PGL_2(\FF_q)$, without having to factor polynomials or find their roots in 
extension fields.

By ``computing a copy of the field $\FF_{q^i}$,'' we mean computing a basis 
$a_1,\ldots,a_i$ for $\FF_{q^i}$ as a vector space over $\FF_q$ and a 
multiplication table for this basis. We choose our basis so that $a_1 = 1$.

\begin{algorithm}
\label{A:odd}
Orbit representatives for the action of $\PGL_2(\FF_q)$ on the degree-$n$ places
of $\Pover{\FF_q}$, where $n\ge 5$ is a fixed odd integer.
\begin{alglist}
\algin  A prime power $q$.
\algout A sequence of monic irreducible polynomials in $\FF_q[x,y]$ of 
        degree~$n$ whose associated divisors form a complete set of unique
        representatives for the action of $\PGL_2(\FF_q)$ on the degree-$n$
        places of $\Pover{\FF_q}$.
\item \label{oddone}
      Compute copies of $\FF_{q^i}$ for $1\le i\le n$ so that we can represent
      divisors as described above.
\item \label{oddtwo}
      Choose an efficiently computable total ordering $<$ on the set of rational
      functions on  $\Pover{\FF_q}$.
\item \label{oddtwopointfive}
      Output the places listed in Theorem~\textup{\ref{T:FF1}} for the given
      values of $n$ and~$q$.
\item \label{oddthree}      
      For every orbit representative $D$ for the action of $\PGL_2(\FF_q)$ on
      the effective divisors of $\Pover{\FF_q}$ of degree at least $3$ and at
      most $(n+1)/2$, do:
      \begin{algsublist}
      \item \label{oddthreea} 
            For every function $F$ obtained from Proposition~\ref{P:fixedpoints}
            applied to $D$, do:
            \begin{algsubsublist}
            \item \label{oddthreeai}
                  For every element $\Gamma$ of $\PGL_2(\FF_q)$ that fixes~$D$,
                  compare $\Gamma\circ F\circ\Gamma^{-1}$ to $F$ in the ordering 
                  from step~\ref{oddtwo}. If $\Gamma\circ F\circ\Gamma^{-1} < F$ 
                  for any such $\Gamma$, then skip to the next value of $F$.
                  (See the proof of Theorem~\ref{T:odd} for more details about 
                  this step.)
            \item \label{oddthreeaii}
                  Compute the $n$-fold iterate $F^{(n)}$ of $F$ and set $g$ to
                  be the numerator of the rational function $x/y - F^{(n)}$.
            \item \label{oddthreeaiii}
                  Set $L$ to be an empty list.
            \item \label{oddthreeaiv}
                  For every monic irreducible degree-$n$ factor $f$ of~$g$,
                  check to see whether $F$ is the Frobenius function for $f$. If
                  so, append the ordered pair $(\Cross f, f)$ to $L$.
            \item \label{oddthreeav}
                  Sort $L$. For every pair $(\Cross f, f)$ on the list such that
                  $\Cross f$ does not appear as the first element of an earlier
                  pair, output the place associated to $f$.
            \end{algsubsublist}
      \end{algsublist}
\end{alglist}
\end{algorithm}

\begin{theorem}
\label{T:odd}
Algorithm~\textup{\ref{A:odd}} outputs a complete set of unique representatives
for the action of $\PGL_2(\FF_q)$ on the degree-$n$ places of $\Pover{\FF_q}$.
With step~\textup{\ref{oddthree}\ref{oddthreea}\ref{oddthreeai}} implemented as
described below, it runs in time $\Otilde(q^{n-3})$ and requires $O(\log q)$
space.
\end{theorem}

\begin{proof}
First we prove that the output is correct: We must show that every orbit of 
$\PGL_2(\FF_q)$ acting on the degree-$n$ places of $\Pover{\FF_q}$ has exactly
one representative included in the output of the algorithm.

The orbits of places whose Frobenius divisors have degree~$2$ will appear in the
output from step~\ref{oddtwopointfive}.

Consider an orbit $\CP$ of places whose Frobenius divisors have degree at 
least~$3$. Among the Frobenius divisors of the polynomials representing elements
of $\CP$, there is exactly one that appears as a divisor $D$ in 
step~\ref{oddthree}. Now consider the set $S$ elements of $\CP$ whose Frobenius
divisors are equal to $D$, and let $\CF$ be the set of Frobenius functions of
elements of~$S$. Exactly one element of $\CF$ satisfies the condition checked in
step~\ref{oddthree}\ref{oddthreea}\ref{oddthreeai}: namely, the smallest element
$F$ of $\CF$ under the ordering chosen in step~\ref{oddtwo}.

Now we narrow our consideration to the elements of $\CP$ whose Frobenius
functions are equal to this~$F$. Each such element is represented by a unique
monic irreducible homogenous polynomial~$f$. Let $\alpha$ be a root of~$f$.
Since $F$ is the Frobenius function for~$f$ we have  
$\proj{\alpha^q}{1} = F(\proj{\alpha}{1})$, and it follows that 
$\proj{\alpha^{q^i}}{1} = F^{(i)}(\proj{\alpha}{1})$ for every~$i$. In 
particular, 
\[
\proj{\alpha}{1} = \proj{\alpha^{q^n}}{1} = F^{(n)}(\proj{\alpha}{1})\,,
\]
so $\proj{\alpha}{1}$ is a fixed point of $F^{(n)}$, and $f$ must divide the
numerator of the polynomial $g$ defined in 
step~\ref{oddthree}\ref{oddthreea}\ref{oddthreeaii}.
Therefore, $f$ occurs as one of the polynomials in 
step~\ref{oddthree}\ref{oddthreea}\ref{oddthreeaiv} for which $(\Cross f, f)$ is
appended to~$L$.

We know that the algorithm will output the place associated to some polynomial
$f'$ with $\Cross f' = \Cross f$. From Theorem~\ref{T:cross}, we know that this
place is in the $\PGL_2(\FF_q)$ orbit of $\CP$. This shows that at least one
representative from each $\PGL_2(\FF_q)$ orbit of places is included in the
output; it remains to show that no orbit is represented more than once.

Suppose $f'$ is an irreducible homogenous polynomial in the same $\PGL_2(\FF_q)$
orbit as~$f$, say $f' = \Gamma(f)$. We would like to show that if $f'\ne f$ then
$f'$ will not appear in the output of the algorithm. To see this is the case, we
let $F'$ be the Frobenius function for $f'$, so that 
$F' = \Gamma\circ F\circ \Gamma^{-1}$ and $\Phi(F') = \Gamma(\Phi(F))$, by 
Lemma~\ref{L:PGLequivariance}. If $\Phi(F')\ne \Phi(F)$ then $f'$ will not
appear in the output, because in step~\ref{oddthree} we choose only one
representative $D$ for each $\PGL_2(\FF_q)$ orbit of divisors. So let us assume
that $\Phi(F') =  \Phi(F)$; that is, we assume that $\Gamma$ fixes 
$D\colonequals\Phi(F)$.

This means that $\Gamma$ is one of the elements of $\PGL_2(\FF_q)$ that we
consider in step~\ref{oddthree}\ref{oddthreea}\ref{oddthreeai}. So in order for 
$f'$ to appear in the output, we must have $F'=F$. This means that $f'$ and $f$
will both appear on the list $L$ produced in 
steps~\ref{oddthree}\ref{oddthreea}\ref{oddthreeaiii} 
and~\ref{oddthree}\ref{oddthreea}\ref{oddthreeaiv}. So after we sort $L$ in
step~\ref{oddthree}\ref{oddthreea}\ref{oddthreeav}, only the smaller of $f$ and
$f'$ will be output. This shows that no orbit is represented more than once in 
our output, so the output is correct.

To analyze the time and space requirements of the algorithm, we must say a
little more about how step~\ref{oddthree}\ref{oddthreea}\ref{oddthreeai} can be
implemented; in particular, we must specify how to compute the 
$\PGL_2(\FF_q)$-stabilizer of~$D$. So consider an effective divisor $D$ of 
degree at least $3$ and at most~$(n+1)/2$.

Let $S$ denote the support of the divisor $D$. If the degree of $S$ is at 
least~$3$, then computing the stabilizer is straightforward, and its size is
bounded by $(\frac{n+1}{2})!$, which means that it is $O(1)$. So let us turn to
the case where $S$ has degree~$1$ or~$2$. First we list the divisors we must 
consider.

Let $P_\infty$ be the place at infinity, let $P_0$ be the place at~$0$, and let
$P_2$ be the place corresponding to an irreducible quadratic $x^2 - uxy + vy^2$,
which we fix for the course of this discussion. If the degree of $S$ is $1$ 
or~$2$, then we may assume that $S$ is either $P_\infty$, $P_\infty + P_0$, 
or~$P_2$. We then find that the divisors with supports of degree $1$ or $2$ that
we must consider, and their stabilizers, are as in 
Table~\ref{table:divisorsandstabilizers}.

\begin{table}
\caption{Representative effective divisors of $\Pover{\FF_q}$ of degree between
$3$ and $(n+1)/2$ with supports of degree at most $2$, together with their
$\PGL_2(\FF_q)$ stabilizers. Here $P_\infty$ and $P_0$ are the places at
$\infty$ and~$0$, respectively, and $P_2$ is the place corresponding to an 
irreducible quadratic $x^2 - uxy + vy^2$.}
\label{table:divisorsandstabilizers}
\centering
\begin{tabular}{l l l}
\toprule
Divisor                       & Conditions                     & Stabilizer\\
\midrule
$m P_\infty$                  & $3\le m\le \frac{n+1}{2}$      & $\bigl\{ \twobytwo{a}{b}{0}{1} \ \vert \ a \in \FF_q^\times, b\in\FF_q \bigr\}$ \\[2.5ex]
$m_\infty P_\infty + m_0 P_0$ & $1\le m_0 < m_\infty$          & $\bigl\{ \twobytwo{a}{0}{0}{1} \ \vert \ a \in \FF_q^\times \bigr\}$ \\
                              & $m_\infty+m_0\le\frac{n+1}{2}$ & \\[2ex]
$m P_\infty + m P_0$          & $2\le m\le\frac{n+1}{4}$       & $\bigl\{ \twobytwo{a}{0}{0}{1} \ \vert \ a \in \FF_q^\times \bigr\} \cup \bigl\{ \twobytwo{0}{a}{1}{0} \ \vert \ a \in \FF_q^\times \bigr\}$ \\[2.5ex]
$m P_2$                       & $2\le m\le\frac{n+1}{4}$       & $\bigl\{ \twobytwo{a}{-bv\phantom{+a}}{b}{-bu + a} \ \vert \ \proj{a}{b} \in \PP^1(\FF_q)\bigr\}$ \\[1ex]
                              &                                & \qquad\qquad$\cup \ \bigl\{ \twobytwo{a}{-au+bv}{b}{-a\phantom{u+bv}} \ \vert \ \proj{a}{b} \in \PP^1(\FF_q)\bigr\}$ \\
\bottomrule
\end{tabular}
\end{table}

Now we can analyze the time taken by the algorithm. The total time to compute
the orbit representatives $D$ in step~\ref{oddthree} is~$\Otilde(q^{(n-5)/2})$,
by the global hypotheses of this section, and there are $O(q^{(n-5)/2})$ such
divisors. To each $D$ there are associated at most $q^{(n-1)/2}$ functions~$F$, 
by Corollary~\ref{C:Ffromfixedpoints}.

First we consider the divisors whose support has degree at least~$3$. There are
$O(q^{(n-5)/2})$ of these divisors, and each gives rise to at most $q^{(n-1)/2}$
possible Frobenius functions~$F$. The total time to process all of these
divisors $D$ is therefore $\Otilde(q^{n-3})$.

Next we consider the divisors with support of degree $1$ or~$2$. If $D$ is one
of the $(n-3)/2$ divisors that appears in the first row of 
Table~\ref{table:divisorsandstabilizers}, there are at most $q^{(n-1)/2}$
associated functions $F$, and it will take time $\Otilde(q^2)$ to process each
one. Therefore, the total time to process all of these $D$ is
$\Otilde(q^{(n+3)/2})$. 

For the $O(n^2)$ divisors from the other rows of the table, there are again at
most $q^{(n-1)/2}$ associated functions $F$, and it takes time $\Otilde(q)$ to 
process each~$F$. Therefore, the total time to process all of these $D$ is
$\Otilde(q^{(n+1)/2})$. 

If $n\ge 9$, then $(n+3)/2\le n-3$, so the time it takes to process these
special divisors falls within our desired upper bound. For $n=7$, the time for
divisors with support of degree~$2$ falls within our desired upper bound, but
the time for divisors with support of degree~$1$ does not. And for $n=5$, even
the divisors with support of degree~$2$ take too long for us. So we must
consider the cases $n=5$ and $n=7$ separately, and modify our algorithm for
these cases.

Suppose $n=5$. We see that there are only two divisors we must consider, 
namely $3 P_\infty$ and $2 P_\infty + P_0$. Let us start with $D = 3 P_\infty$.

From Proposition~\ref{P:fixedpoints} we compute that the rational functions~$F$
with $\Phi(F) = D$ are the functions 
\[
F \colonequals \frac{x^2 + rxy + sy^2}{xy + ry^2}
\]
for $r\in \FF_q$ and $s\in \FF_q^\times$. If $q$ is odd, pick a nonsquare
$\nu\in \FF_q$. No matter the parity of $q$, choose $a\in \FF_q^\times$ so that
$sa^2$ is either $1$ or~$\nu$, set $b=ar$, and consider the element
$\Gamma = \twobytwo{a}{b}{0}{1}$ of the stabilizer of $D$. We compute that 
\[
\Gamma\circ F\circ\Gamma^{-1} = \frac{x^2 + sa^2 y^2}{xy}\,.
\]
For odd $q$, it is easy to check that the two functions $(x^2 + y^2)/(xy)$ and
$(x^2 + \nu y^2)/(xy)$ are not in the same $\PGL_2(\FF_q)$ orbit. Thus, to 
process the divisor $D = 3 P_\infty$ we need only consider at most $2$ specific
functions, each one taking time $\Otilde(1)$ to process.

Now suppose $D = 2P_\infty + P_0$. The functions $F$ with $\Phi(F) = D$ are
\[
F \colonequals \frac{x^2 + sxy}{xy + ry^2}
\]
where $s\ne r$ and $r\ne 0$. The stabilizer of $D$ in $\PGL_2(\FF_q)$ consists of the
elements $\Gamma = \twobytwo{a}{0}{0}{1}$ with $a$ nonzero, and we compute that 
\[
\Gamma\circ F\circ\Gamma^{-1} = \frac{x^2 + asxy}{xy + ary^2}\,.
\]
Therefore, up to the stabilizer of $D$, the only functions we have to consider
are $(x^2 + sxy)/(xy + y^2)$ (with $s\ne 1$). The total
time to process all of these functions is $\Otilde(q)$, which is within our
desired bound of $\Otilde(q^2)$.

Now we turn to the case $n=7$. The divisor $3P_\infty$ can be treated as
described above. Thus the only remaining divisor we must treat specially is
$D = 4 P_\infty$. We leave it to the reader to show that if $q$ is odd, then
each orbit of functions $F$ with $\Phi(F) = D$ under the action of the 
stabilizer of $D$ contains exactly one element in the set
\[
\Bigl\{
\frac{x^3 + rxy^2 + ry^3}{x^2y + ry^3} \ \Big\vert \ r\in \FF_q^\times
\Bigr\}
\cup
\Bigl\{
\frac{x^3 + ry^3}{x^2y} \ \Big\vert \ r\in S
\Bigr\}
\]
where $S$ is a set of representatives for $\FF_q^\times/\FF_q^{\times 3}$. The
total time to process these functions is $\Otilde(q)$, which is less than our 
goal of $\Otilde(q^4)$. On the other hand, if $q$ is even and if we choose an 
element $\nu\in\FF_q$ of absolute trace~$1$, then then each orbit of functions
$F$ with $\Phi(F) = D$ under the action of the stabilizer of $D$ contains
exactly one element in the set
\[
\Bigl\{
\frac{x^3 + x^2y + rxy^2 + sy^3}{x^2y + xy^2 + ry^3} \ \Big\vert \ r \in\{0,\nu\}, s\in \FF_q^\times
\Bigr\}
\cup
\Bigl\{
\frac{x^3 + ry^3}{x^2y} \ \Big\vert \ r\in S
\Bigr\}
\]
where $S$ is again a set of representatives for $\FF_q^\times/\FF_q^{\times 3}$.
The total time to process these functions is again $\Otilde(q)$.

We see that in every case, the algorithm can be implemented to take time
$\Otilde(q^{n-3})$. The fact that it requires space $O(\log q)$ is clear.
\end{proof}

\section{Enumerating places of even degree up to 
\texorpdfstring{$\PGL_2$}{PGL2}}
\label{S:simpleeven}

In this section we present an algorithm for enumerating representatives for the
orbits of $\PGL_2(k)$ acting on the places of $\PP^1$ of even degree~$n$ over a
finite field~$k$, using $O(\log q)$ space and $\Otilde(q^{n-3})$ time, where 
$q = \#k$. As in the preceding section, our hypothesis throughout this section
is that we can accomplish this same task for places of every degree~$d$ smaller
than~$n$ --- but we will only need to use the case $d = n/2$ of this hypothesis.

The spirit of the algorithm is similar to that 
of~\cite[Algorithm~7.12]{Howe2024}. In that algorithm we produce a large list of
orbit representatives that includes roughly two entries for each orbit, and then
deduplicate the list by using cross polynomials. The main difference here is
that instead we take much more care in producing the list to begin with, so that
we do not have to use memory to store the list.

The algorithm makes use of the orbit labeling function $\CL$ from
Definition~\ref{D:orbitlabels}.

\begin{algorithm}
\label{A:even}
Orbit representatives for the action of $\PGL_2(\FF_q)$ on the degree-$n$ places
of $\PP^1$ over $\FF_q$, where $n\ge 6$ is a fixed even integer.
\begin{alglist}
\algin  A prime power $q$.
\algout A sequence of monic irreducible polynomials in $\FF_q[x,y]$ of 
        degree~$n$ whose associated divisors form a complete set of unique
        representatives for the action of $\PGL_2(\FF_q)$ on the degree-$n$
        places of $\Pover{\FF_q}$.
\item \label{evenone}
      Choose efficiently computable total orderings $<$ on the set
      $\FF_{q^2}\times \FF_{q^2}\times \FF_{q^2}$ and on the set of polynomials
      $\FF_{q^2}[x]$.
\item \label{eventwo}      
      For every orbit representative $P$ for the action of $\PGL_2(\FF_{q^2})$
      on the places of $\Pover{\FF_{q^2}}$ of degree~$n/2$, do:
      \begin{algsublist}
      \item \label{eventwoa}
            If $\Cross(P^{(q)}) < \Cross(P)$ then skip to the next value of~$P$.
      \item \label{eventwob}
            For every orbit representative $\Gamma$ from 
            Proposition~\ref{P:CosetRepsPGL2} for the action of $\PGL_2(\FF_q)$
            on $\PGL_2(\FF_{q^2})$, do:
            \begin{algsubsublist}
            \item \label{eventwobi}
                  For every $\Psi\in\PGL_2(\FF_{q^2})$ that fixes~$P$, compare 
                  $\CL(\Gamma\Psi)$ to $\CL(\Gamma)$. If 
                  $\CL(\Gamma\Psi)<\CL(\Gamma)$ for any such $\Psi$, skip to the
                  next value of $\Gamma$.
            \item \label{eventwobii}
                  If $\Cross(P^{(q)}) = \Cross(P)$, then for every
                  $\Psi\in\PGL_2(\FF_{q^2})$ that takes $P$ to~$P^{(q)}$, 
                  compare $\CL(\Gamma^{(q)}\Psi)$ to $\CL(\Gamma)$. If 
                  $\CL(\Gamma^{(q)}\Psi)<\CL(\Gamma)$ for any such $\Psi$, skip
                  to the next value of $\Gamma$.   
            \item \label{eventwobiii}
                  Let $f\in\FF_{q^2}[x,y]$ be the monic homogeneous polynomial 
                  associated to the place $P$. If $\Gamma(f)$ does not lie in
                  $\FF_q[x,y]$, output the product of $\Gamma(f)$ with its
                  conjugate $\Gamma^{(q)}(f^{(q)})$.
            \end{algsubsublist}
      \end{algsublist}
\end{alglist}      
\end{algorithm}

\begin{theorem}
\label{T:even}
Algorithm~\textup{\ref{A:even}} outputs a complete set of unique representatives
for the action of $\PGL_2(\FF_q)$ on the degree-$n$ places of $\Pover{\FF_q}$.
It runs in time $\Otilde(q^{n-3})$ and requires $O(\log q)$ space.
\end{theorem}

\begin{proof}
Let $m=n/2$. 
Every monic irreducible degree-$n$ homogeneous polynomial $f\in \FF_q[x,y]$
factors in $\FF_{q^2}[x,y]$ into a product $g g^{(q)}$ of a monic irreducible
degree-$m$ polynomial $g$ with its conjugate, and $g$ is unique up to
conjugation over~$\FF_q$. So what we would like to do is to enumerate the set of
degree-$m$ places of $\Pover{\FF_{q^2}}$ that are \emph{not} the lifts of
degree-$m$ places of $\Pover{\FF_{q}}$, up to the combined action of 
$\PGL_2(\FF_q)$ and $\Gal(\FF_{q^2}/\FF_q)$.

By our induction hypothesis, we already know how to get a complete set of unique
representatives for the action of $\PGL_2(\FF_{q^2})$ on the degree-$m$ places
$Q$ of $\Pover{\FF_{q^2}}$, so all we must do is figure out how to expand the
representative $Q$ for a given $\PGL_2(\FF_{q^2})$ orbit to get representatives 
for the $\PGL_2(\FF_{q})$ orbits without introducing duplicates, and all up to 
the action of $\Gal(\FF_{q^2}/\FF_q)$.

We will certainly obtain representatives for all of the $\PGL_2(\FF_{q})$ orbits
that are in the $\PGL_2(\FF_{q^2})$ orbit of a place $Q$ if we simply consider
the set of places $\{\Gamma(Q)\}$, for the $\Gamma$ listed in 
Proposition~\ref{P:CosetRepsPGL2}. But when will we get multiple
representatives?

We will get multiple representatives exactly when we are in the situation that 
$\Gamma_1(Q) = \Phi\Gamma_2(Q)$, where $\Phi$ is a nontrivial element of
$\PGL_2(\FF_q)$ and where $\Gamma_1$ and $\Gamma_2$ are elements of 
$\PGL_2(\FF_{q^2})$ such that $\Gamma_1\not\in \PGL_2(\FF_q)\cdot \Gamma_2$. We 
see that this can only happen when there is an element $\Psi$ of the 
$\PGL_2(\FF_{q^2})$ stabilizer of $Q$ such that $\Gamma_1 = \Phi\Gamma_2\Psi$; 
in other words, for a given $\Gamma_1$ and $\Gamma_2$ there will only be such an
element $\Phi$ when there is a $\Psi$ in the stabilizer of $Q$ such that 
$\Gamma_1$ and $\Gamma_2\Psi$ are in the same $\PGL_2(\FF_q)$ orbit. We can tell
whether this is the case by comparing $\CL(\Gamma_1)$ with $\CL(\Gamma_2\Psi)$.
In step~\ref{eventwo}\ref{eventwob}\ref{eventwobi}of Algorithm~\ref{A:even}, we 
compare $\CL(\Gamma)$ to $\CL(\Gamma\Psi)$ for all such~$\Psi$, and move on to
the next $\Gamma$ if $\CL(\Gamma)$ is not the smallest of these orbit labels.
This is enough to ensure that we are getting a complete set of unique 
representatives for the action of $\PGL_2(\FF_q)$ on the set of degree-$m$
places of $\Pover{\FF_{q^2}}$.

The only other way we may output more than one representative for the same orbit
is if we have two orbit representatives $P_1$ and $P_2$ from step~\ref{eventwo},
and two elements $\Gamma_1$ and $\Gamma_2$ from 
step~\ref{eventwo}\ref{eventwob}, such that the pairs $(P_1,\Gamma_1)$ and 
$(P_2,\Gamma_2)$ are not equal to one another, but such that the Galois
conjugate of $\Gamma_2(P_2)$ lies in the $\PGL_2(\FF_q)$ orbit of 
$\Gamma_1(P_1)$. This will be the case precisely when 
$\Gamma_1 = \Phi\Gamma_2^{(q)}\Psi$ for some $\Psi\in\PGL_2(\FF_{q^2})$ that 
takes $P_1$ to~$P_2^{(q)}$ and for some $\Phi\in\PGL_2(\FF_q)$. Note that this
can only happen if $\Cross(P_1) = \Cross(P_2^{(q)})$, and in fact does happen in
this case. We avoid this in two ways:

First, we only consider places $P$ such that $\Cross(P) \le \Cross(P^{(q)})$. By
doing so, we further restrict the bad circumstance above to the case where 
$P_1 = P_2$ and where $\Cross(P_1)$ lies in $\FF_q[x]$. This is accomplished by
step~\ref{eventwo}\ref{eventwoa}.

Next, if we are considering a representative $P$ for which $\Cross(P)$ lies in
$\FF_q[x]$, we want to avoid producing output for two different values of 
$\Gamma$, say $\Gamma_1$ and $\Gamma_2$, if $\Gamma_1 = \Phi\Gamma_2^{(q)}\Psi$
for some $\Psi$ that takes $P$ to~$P^{(q)}$.
Step~\ref{eventwo}\ref{eventwob}\ref{eventwobii} ensures that this does not
happen.

This shows that the algorithm does indeed output exactly one representative of
each $\PGL_2(\FF_q)$ orbit of the divisors of degree~$n$ on $\Pover{\FF_q}$. It
is clear that the space needed is $O(\log q)$. All we have left to do is to show
that the algorithm runs in time~$\Otilde(q^{n-3})$. 

There are $O((q^2)^{m-3}) = O(q^{n-6})$ orbit representatives $P$ to consider in
step~\ref{eventwo}, and the total time it takes to generate these 
representatives is $\Otilde(q^{n-6})$. For each such representative~$P$, we need
to consider the $O(q^3)$ possible values of $\Gamma$ from 
Proposition~\ref{P:CosetRepsPGL2}, and for each pair $(P,\Gamma)$ we do 
$\Otilde(1)$ work. All told, this results in $\Otilde(q^{n-3})$ work, as 
claimed.
\end{proof}

\section{Enumerating hyperelliptic curves in Magma}
\label{S:implementation}

Using the ideas of this paper and of~\cite{Howe2024}, we have implemented
algorithms to enumerate hyperelliptic curves of genus 2, 3, and 4 over finite
fields of odd characteristic in Magma~\cite{BosmaCannonPlayoust1997}. 
Our implementations can be found at
\url{https://github.com/everetthowe/hyperelliptic}.

The results of this paper show that enumerating
hyperelliptic curves of genus $g$ over $\FF_q$ can be done in
$\Otilde(q^{2g-1})$  time using $O(\log q)$ space. Using so little
memory is important from a theoretical perspective, but for our implementations
we decided to allow ourselves $O(q^{g-1})$ space. 

Table~\ref{table} shows the time required to enumerate hyperelliptic 
curves of genus $2$, $3$, and $4$ over various finite fields $\FF_q$ using
the programs in \texttt{Hyperelliptic2.magma} (v.~2.00), 
\texttt{Hyperelliptic3.magma} (v.~2.01), and 
\texttt{Hyperelliptic4.magma} (v.~0.9), running in Magma V2.28-20 on
one core of an Apple M4 Max with 64GB of RAM. We ran our enumeration programs
without storing or outputting the resulting curves.

If our enumeration of genus-$g$ hyperelliptic curves over $\FF_q$ takes time~$t$,
we let $r_g(q)$ denote the quantity $10^5 \cdot t / q^{2g-1}$, so that for 
fixed $g$ the quantity $r_g(q)$ should be bounded above by a power of 
$\log q$. (The scaling factor $10^5$ was chosen for convenience.) Table~\ref{table} 
gives the values of $r_g(q)$ for the examples we compute.
Note the jump in $r_2(q)$ between $q=1021$ and
$q=1031$; this is almost certainly explained by the fact that Magma uses
Zech logarithms to speed up arithmetic in $\FF_{q^2}$ for $q\le 1024$ but not for 
larger~$q$~\cite[\S22.1.1]{BosmaCannonEtAl2025}.

\begin{table}
\caption{Time (in seconds) to enumerate hyperelliptic curves of genus $2$, $3$, and
4 over various finite fields $\FF_q$. The notation $r_g(q)$ is explained in the
text of the article, as is the description of the hardware and software used.}
\label{table}
\centering
\begin{tabular}{r r c  c  r r c  c  r r c}
\toprule
\multicolumn{3}{c}{Genus 2} && \multicolumn{3}{c}{Genus 3}&& \multicolumn{3}{c}{Genus 4}\\
\cmidrule(lr){1-3}  \cmidrule(lr){5-7}\cmidrule(lr){9-11}
$q$ &\multicolumn{1}{c}{Time (sec)} & $r_2(q)$ &&
$q$ &\multicolumn{1}{c}{Time (sec)} & $r_3(q)$ &&
$q$ &\multicolumn{1}{c}{Time (sec)} & $r_4(q)$\\
\cmidrule[\lightrulewidth]{1-3}  \cmidrule[\lightrulewidth]{5-7}\cmidrule[\lightrulewidth]{9-11}
$  31$& $     0.92$ &  $3.09$ && $ 7$ & $    0.36$ &  $2.14$ && $ 5$ & $    2.93$ & $3.75$ \\
$  61$& $     5.11$ &  $2.25$ && $11$ & $    2.27$ &  $1.41$ && $ 7$ & $   25.28$ & $3.07$ \\
$ 127$& $    34.72$ &  $1.69$ && $17$ & $   16.18$ &  $1.14$ && $ 9$ & $  160\zz$ & $3.36$ \\
$ 257$& $   285\zz$ &  $1.68$ && $23$ & $   60.27$ &  $0.93$ && $11$ & $  607\zz$ & $3.12$ \\
$ 509$& $  2002\zz$ &  $1.52$ && $31$ & $  241\zz$ &  $0.84$ && $13$ & $ 2061\zz$ & $3.29$ \\
$1021$& $ 16041\zz$ &  $1.51$ && $43$ & $ 1259\zz$ &  $0.86$ && $17$ & $11712\zz$ & $2.85$ \\
$1031$& $ 22002\zz$ &  $2.01$ && $61$ & $ 7545\zz$ &  $0.89$ && $19$ & $24898\zz$ & $2.79$ \\
$2039$& $174958\zz$ &  $2.06$ && $89$ & $50721\zz$ &  $0.91$ && $23$ & $91693\zz$ & $2.69$ \\
\bottomrule
\end{tabular}
\end{table}

Another way to enumerate hyperelliptic curves of genus~$2$ over $\FF_q$
in Magma using little memory would be to compute

\centerline{\texttt{Twists(HyperellipticCurveFromG2Invariants([a,b,c]))}}
\noindent
for every triple $(a,b,c)$ of elements of~$\FF_q$. For a given~$q$, we can
estimate the time this method would take by computing the expression above
for, say, $10^5$ random triples $(a,b,c)$, and then multiplying the time this
computation takes by $q^3/10^5$. These estimates show that 
for $q<1024$ our method is approximately $80$ times faster than this alternative,
while for $q > 1024$ our method seems to be about $60$ times faster. 

There is also an alternative way of enumerating hyperelliptic curves of
genus~$3$ over $\FF_q$ in Magma using little memory: One can apply the
Magma command

\centerline{\texttt{TwistedHyperellipticPolynomialsFromShiodaInvariants(S)}}
\noindent
to all Shioda invariants $S$ with nonzero discriminants, obtained by computing

\centerline{\texttt{ShiodaAlgebraicInvariants(v~:~ratsolve:=true)}}
\noindent
for every element $v$ of the $5$-dimensional weighted projective space $V$ over
$\FF_q$ with weights $[2,3,4,5,6,7]$ and discarding the Shioda invariants
with discriminant~$0$. We can estimate the time this method would take by sampling
as before. We find that for the $q\ge 31$ in Table~\ref{table},
our method for enumerating genus-$3$ hyperelliptic  curves is about $280$ times
faster than this alternative method.

We are not aware of any other implementations of algorithms to enumerate
genus-$4$ hyperelliptic curves over finite fields, so there are no timing
comparisons to make for the genus-$4$ case.

\bibliographystyle{hplaindoi}
\bibliography{places}

\end{document}